\documentclass[11pt]{article}
\usepackage[margin=1.25in]{geometry}

\usepackage{natbib} \bibpunct{(}{)}{;}{a}{}{,}
\usepackage{hyperref}
\usepackage{amsmath,amssymb,amsthm}
\usepackage{enumerate}
\usepackage{graphicx}
\usepackage{subfig}
\usepackage{authblk}

\newenvironment{delayedproof}[1]{\par\noindent{\bf #1}}{\hfill$\square$\\[2mm]}

\newtheorem{proposition}{Proposition}
\newtheorem{lemma}{Lemma}
\newtheorem{theorem}{Theorem}
\newtheorem{example}{Example}
\newtheorem{remark}{Remark}
\newtheorem{definition}{Definition}

\renewcommand{\AA}[1]{A^{(#1)}}

\newcommand{\nub}[1]{(\nu_{#1} + b_{#1}^2)}

\newcommand{\mxbernsymbol}{\eta}

\renewcommand{\Pr}{\mathbb{P}}
\newcommand{\E}{\mathbb{E}}

\newcommand{\R}{\mathbb{R}}

\newcommand{\calI}{\mathcal{I}}

\newcommand{\calM}{\mathcal{M}}
\newcommand{\calN}{\mathcal{N}}
\newcommand{\calP}{\mathcal{P}}

\newcommand{\calS}{\mathcal{S}}
\newcommand{\symmetricmxs}{\calS}

\newcommand{\Ahat}{\hat{A}}

\newcommand{\Phat}{\hat{P}}

\newcommand{\Xhat}{\hat{X}}

\newcommand{\what}{\hat{w}}
\newcommand{\thetahat}{\hat{\theta}}
\newcommand{\rhohat}{\hat{\rho}}

\newcommand{\tauhat}{\hat{\tau}}

\newcommand{\chat}{\hat{c}}

\newcommand{\Aopt}{\mathring{A}}
\newcommand{\Vopt}{\mathring{V}}
\newcommand{\Xopt}{\mathring{X}}
\newcommand{\wopt}{\mathring{w}}

\newcommand{\Abar}{\bar{A}}
\newcommand{\Pbar}{\bar{P}}
\newcommand{\Xbar}{\bar{X}}
\newcommand{\cbar}{\bar{c}}
\newcommand{\nubar}{\bar{\nu}}
\newcommand{\bbar}{\bar{b}}

\newcommand{\Atilde}{\tilde{A}}
\newcommand{\UAtilde}{U_{\Atilde}}
\newcommand{\SAtilde}{S_{\Atilde}}
\newcommand{\Otilde}{\tilde{O}}

\newcommand{\Xtilde}{\tilde{X}}

\newcommand{\rhotilde}{\tilde{\rho}}

\newcommand{\bB}{\mathbf{B}}
\newcommand{\bZ}{\mathbf{Z}}

\newcommand{\bstar}{b^*}
\newcommand{\nustar}{\nu^*}
\newcommand{\nmin}{n_{\min}}

\newcommand{\SA}{S_A}
\newcommand{\SM}{S_M}
\newcommand{\SP}{S_P}
\newcommand{\UP}{U_P}
\newcommand{\UA}{U_A}
\newcommand{\UM}{U_M}

\newcommand{\ASE}{\operatorname{ASE}}

\newcommand{\rank}{\operatorname{rank}}

\newcommand{\tti}{2,\infty}

\newcommand{\JsubSBM}{\operatorname{J-\Gamma-SBM}}
\newcommand{\KL}{\operatorname{KL}}

\newcommand{\Phatwtd}{\Phat}
\newcommand{\Phatunif}{\Pbar}

\newcommand{\ivec}{\vec{i}}
\newcommand{\jvec}{\vec{j}}
\newcommand{\indicator}{\mathbb{I}}

\newcommand{\indsim}{\stackrel{\text{ind.}}{\sim}}

\begin{document}

\title{Recovering shared structure from multiple networks with unknown edge distributions}

\author{Keith Levin}
\affil{Department of Statistics\\
       University of Wisconsin-Madison}
\author{Asad Lodhia}
\affil{Unaffiliated}
\author{Elizaveta Levina}
\affil{Department of Statistics\\
       University of Michigan}

\maketitle

\begin{abstract}
In increasingly many settings, data sets consist of multiple samples from a population of networks, with vertices aligned across these networks. For example, brain connectivity networks in neuroscience consist of measures of interaction between brain regions that have been aligned to a common template. We consider the setting where the observed networks have a shared expectation, but may differ in the noise structure on their edges. Our approach exploits the shared mean structure to denoise edge-level measurements of the observed networks and estimate the underlying population-level parameters. We also explore the extent to which edge-level errors influence estimation and downstream inference. We establish a finite-sample concentration inequality for the low-rank eigenvalue truncation of a random weighted adjacency matrix that may be of independent interest. The proposed approach is illustrated on synthetic networks and on data from an fMRI study of schizophrenia.
\end{abstract}

\section{Introduction}
\label{sec:intro}
Many modern applications require simultaneous analysis of multiple networks, often with the goal of identifying structure
that is shared across multiple networks.
In the social sciences, this may correspond to some common underlying
structure that appears, for example, in different friendship networks
across high schools.
In biology, one may be interested in identifying the extent to which different
organisms' protein-protein interaction networks display a similar structure.
In neuroscience, one may wish to identify common patterns across
multiple subjects' brains in an imaging study.
This last application in particular is easily abstracted to the situation
where one observes a collection of independent graphs on the same vertex set,
as there are well-established and widely used algorithms that map locations
in individual brains onto a common atlas of so-called
regions of interest (ROIs),
such as the one developed by \cite{PowerETAL2011}.
Thus, the assumption of vertex correspondence across graphs
is especially common in multiple network analysis
for neuroimaging applications; see for example
\cite{LevAthTanLyzPri2017,ArrKesLevTay2017}.
A common approach to these problems
is to treat the individual networks as independent
noisy realizations of some shared structure such as a stochastic block model
\citep{LeLevLev2018} or a low-rank model \citep{TanKetVogPriSus2016}.
The goal is then to recover this underlying shared  structure.
The importance of analyzing brain data
from multiple subjects simultaneously has spurred a particularly
active line of work on this problem in neuroimaging.
As a result, we primarily focus on this literature,
but stress that the general problem of
recovering a shared structure from multiple networks
(or, more generally, multiple matrices) has applications to many other domains.

To date, most techniques for multiple-network analysis have assumed that
the observed networks come from the same distribution,
and typically have binary edges, a restrictive assumption in many settings.
For example, in social networks, edge weights may represent the strengths of
friendships, in gene expression networks, edge weights represent the extent to which
pairs of genes are co-expressed \citep{ZhaHor2005}, and in neuroimaging, edge weights represent the strength of connectivity
between brain regions of interest.
Substantial information is lost if these weights are truncated to binary;
see, for example, \cite{AicJacCla2015}.
The shared noise distribution is also a restrictive assumption,
since while we may reasonably expect 
that some population-level structure is shared across networks,
network-level variation is likely to be heterogeneous.
Specifically in neuroimaging, a lot of 
subject-level variation comes from head motion or deviations of the individual subject's brain from the common atlas, and there is no reason to expect these to be homogeneous.     There are a number of different pipelines
in use for reducing this type of noise in fMRI data
\citep[see, for example,][for a discussion]{CiricETAL2017},
but all introduce artifacts of one kind or another,
which we model here as potentially heterogeneous edge noise.

In this paper, we develop techniques for analyzing multiple networks without
these two assumptions.
In particular, we allow for weighted edges and heterogeneous
noise distributions, and study how to estimate the underlying population mean.
Under these conditions, the simple arithmetic mean of weighted
graphs is likely to be sub-optimal, as networks with higher noise levels
will contribute as much as those with less noise, and one would intuitively expect an estimate that
takes noise levels into account to perform better.
While there are a number of possible matrix means we might consider
\citep[see, for example, those described in][]{Bhatia2007}, 
in this paper, we focus on the case of weighted arithmetic means
of networks. That is, letting $\AA{1},\dots,\AA{N} \in \R^{n \times n}$
be the adjacency matrices of independent graphs on the same vertex set,
we are interested in estimators of the form
$\sum_{s=1}^N \what_s \AA{s}$, where $\{ \what_s \}_{s=1}^N$ are non-negative,
data-dependent weights summing up to $1$.


There have been a number of papers written in recent
years related to analysis of multiple vertex-aligned networks.
Motivated by brain imaging applications similar to those discussed here,
\cite{TanKetVogPriSus2016} 
considered the problem of estimating a low-rank population matrix,
 when graphs are drawn i.i.d.\ from a random dot product graph model
\citep{RDPGsurvey}, and investigated the asymptotic relative efficiency
of a low-rank approximation of the sample mean of these observed graphs
compared to the graph sample mean itself.   
\cite{LevAthTanLyzPri2017} considered the problem of analyzing multiple
vertex-aligned graphs,
and devised a method to compare geometric representations of graphs,
typically called embeddings in the literature,
for the purpose of exploratory data analysis and hypothesis testing,
focused particularly on comparing vertices across graphs.  
In a similar spirit, \cite{WanArrVogPri2017}
considered the problem of embedding multiple binary graphs whose adjacency
matrices (approximately) share eigenspaces,
while possibly differing in their eigenvalues.
All of these papers assume binary networks and identical noise
distributions on edges, in contrast to the setting we study here.  

\cite{EynKovBroGlaBro2015} developed a technique for analyzing multiple
manifolds by (approximately) simultaneously
diagonalizing a collection of graph Laplacians.
Like our work, the technique in \cite{EynKovBroGlaBro2015} aims
to recover spectral information shared across multiple observed graphs, 
but differs in that the authors work with weighted similarity graphs that arise
from data lying on a manifold, and
derive a perturbation bound instead of applying a specific
statistical or probabilistic model.
The authors also require that the population graph Laplacian
have a simple spectrum, an additional assumption we do not need. 

A few recent papers have considered the problem of analyzing
multiple networks generated from stochastic block models with the same
community structure, but possibly different 
connection probability matrices 
\citep{TanLuDhi2009,DonFroVanNef2014,HanXuAir2015,PauChe2016,BhaCha2018}. 
Similar approaches have been developed for time-varying networks,
where it is assumed that the connection probability may change over time,
but community structure is constant or only slowly varying
\citep{XuHer2014}.
Our setting is distinct from this line of work,
since we assume a general shared structure with varying distribution
of edge noise, and do not require edges to be binary.

\cite{TanTanVogPri2017} considered the problem of estimating shared low-rank
structure based on a sample of networks under the setting where individual
edges are drawn from contaminated distributions \citep{Huber1964}.
The paper compares the theoretical guarantees of estimates
based on edge-wise (non-robust) maximum likelihood estimation,
edge-wise robust maximum likelihood estimation \citep{FerYan2010},
and eigenvalue truncations of both.
The present work does not focus on robustness, and as a result is largely not  
comparable to \cite{TanTanVogPri2017}, although our procedures can be made robust in a similar fashion if desirable.


The remainder of this paper is organized as follows:
In Section~\ref{sec:problemsetup}, we give a formal description of the
problem and present necessary background.
Our main results are presented in
Section~\ref{sec:method}, with shows the optimality of a certain weighted network
average,  and Section~\ref{sec:estimate}, which proposes an algorithm to estimate these optimal weights from data.
Section~\ref{sec:perfect} briefly considers the application to community detection
and estimation in the multiple-network setting.
Section~\ref{sec:experiments} explores the effectiveness of our method on
simulated data as well as on a fMRI dataset from a study
comparing schizophrenic and healthy patients.
We conclude with a brief discussion in Section~\ref{sec:discussion}
summarizing our results and sketching directions for future work.

\section{Problem Setup and Notation}
\label{sec:problemsetup}
Throughout this paper, we assume that
we observe $N$ undirected graphs 
each on $n$ vertices with corresponding adjacency matrices,
$\AA{1},\AA{2},\dots,\AA{N} \in \R^{n \times n}$.
We will  refer to the $s$-th graph and its adjacency matrix $\AA{s}$ interchangeably.
Our key assumption is that the graphs are drawn independently with shared  expectation $\E \AA{s} = P \in \R^{n \times n}$, for all
$s \in [N] = \{1,2,\dots,N\}$.
Throughout, we will denote the rank of $P$ by $d = \rank P$.  All results will depend on $d$, but we do not make a low-rank assumption;  all our results are finite sample and thus valid for any $d \le n$.   
We assume that the vertices are aligned across
the graphs, in the sense that the $i$-th vertex in graph
$\AA{s}$ is identifiable with the $i$-th vertex in graph
$\AA{t}$ for all $i \in [n]$ and $s,t \in [N]$.
As a motivating example,
consider the case where the observed graphs are obtained from fMRI neuroimaging
of $N$ patients.
In such a setting, the $n$ vertices of each graph correspond to brain regions
of interest (ROIs), identified based on alignment to a common template
\citep[e.g.,][]{PowerETAL2011}.
This common alignment ensures that the $i$-th vertex in each of these
connectomes corresponds to the same anatomical region,
and thus this $i$-th vertex can be sensibly identified across graphs.
If the $N$ patients belong to a common population
(e.g., shared disease status), it is reasonable to expect that these networks
exhibit a shared structure.
In this work,
we take this shared structure to be a common expectation.
Note that for the case of $N =1$ a low-rank or another structural assumption would have to be made on $P$ in order to enable estimation, since otherwise we would only have one observation per parameter.
With $N > 1$ observed networks, this is not strictly necessary, but a low-rank or some other structural assumption would certainly enable better estimation, just like it does for the $N=1$ case.


Crucially for the purposes of this paper,
while the expectation $P$ is shared across the graphs,
we allow for each graph to exhibit different edge noise structure.
That is, for each $s \in [N]$ and $i,j \in [n]$,
$(\AA{s}-P)_{ij}$ has mean 0 but otherwise arbitrary
distribution $F = F_{s,ij}$, which may depend both on the subject $s$
and the specific edge $(i,j)$.
In the motivating neuroimaging application, this corresponds to the fact that  high-level anatomical and functional structure is likely common across patients,
but measurement noise is likely subject-specific, and edge heterogeneity may also result from individual differences.  
For simplicity of notation, we allow for self-loops, 
i.e., treat $\AA{s}_{ii}$ exactly the same way as the off-diagonal entries; 
self-loops
are generally a moot point for asymptotics,  since they make a negligible $O(n)$ contribution compared to the
$O(n^2)$ off-diagonal entries. 
Throughout, we assume that all parameters, including the number of networks
$N$, can depend on the number of vertices $n$,
though we mostly suppress this dependence for ease of reading.
We write $C$ for a generic positive constant,
not depending on $n$, whose value may change from one line to the next.   

Before proceeding, we pause to establish notation.
For an integer $k$, we write $[k]$ for the set  $\{1,2,\dots,k\}$.
For a vector $v$, we write $\| v \|$ for the Euclidean norm of $v$.
For a matrix $M$,  $\| M \|$ denotes the spectral norm, 
$\| M \|_F$ the Frobenius norm,
and $\| M \|_{\tti}$ the $(2,\infty)$ norm,
$\| M \|_{\tti} = \sup_{v:\|v\|=1} \| Mv \|_{\infty}$,
where $\| v \|_\infty = \max_i |v_i|$.
For a positive semidefinite matrix $M$,
we write $\kappa(M)$ for the ratio of the largest eigenvalue of $M$ to its largest {\em non-zero} eigenvalue.
We use standard Landau notation to $O(\cdot),o(\cdot),\Omega(\cdot),$ and $\omega(\cdot)$
to denote growth rates.
For example, $g(n) = O(f(n))$ as $n \rightarrow \infty$  means that $|g(n)| < Cf(n)$  for some constant $C$ and all $n > n_0$, $g(n) = \Omega(f(n))$ means $g(n) > Cf(n)$, and so on.  We use $\Otilde$ to denote growth rate up to log-factors, as in, for example, $n \log^2 n = \Otilde( n )$.    In a slight abuse of the term, 
we say that an event $E_n$ occurs {\em with high probability} (w.h.p.) if $\Pr[ E_n^c ] \le C n^{-(1+\epsilon)}$ for some constant $\epsilon > 0$. 
This definition allows us to state our results as finite-sample bounds while immediately implying asymptotic results of the form ``with probability $1$, event $B_n$ occurs for at most finitely many $n$'' by applying the Borel-Cantelli lemma.

%
\subsection{Motivating Examples} \label{subsec:motivating}

We now present a few examples that satisfy our model assumptions, in order of increasing generality.   In all cases, the question is how to optimally recover the underlying shared expectation $P$.  We begin with one of the simplest possible settings under our model.  

\begin{example}[Normal measurement errors with subject-specific variance]
\label{ex:gaussian}
~Assume that for each $s=1,2,\dots,N$, 
$\{ (\AA{s}-P)_{ij} : 1 \le i \le j \le n \}$ are independent
$\calN(0,\rho_s)$.
\end{example}

A weaker assumption on the edge measurement errors would be to replace
the specific distributional assumption that
$(\AA{s}-P)_{ij} \sim \calN(0,\rho_s)$
with a more general tail bound assumption,
such as sub-Gaussian or sub-gamma errors.
We refer the reader to Appendix~\ref{apx:taildefs} for the definition
and a few basic properties of sub-Gaussian and sub-gamma random variables,
or to \cite{BLM} for a more substantial discussion.

\begin{example}[Sub-gamma measurement errors with subject-specific parameter]
\label{ex:subgammasubject}
~Assume that for each $s = 1,2,\dots,N$, 
$\{ (\AA{s}-P)_{ij} : 1 \le i \le j \le n \}$ are independent, mean $0$,
sub-gamma with parameters $(\nu_s,b_s)$.

\end{example}


Note that the assumptions of Example \ref{ex:subgammasubject} do not require the edges to be identically distributed within a network.
We can further relax the sub-gamma assumption to allow for edge-specific
tail parameters rather than having a single tail parameter for each subject.

\begin{example}[Sub-gamma errors with subject- and edge-specific parameters]
\label{ex:subgamma}
~Assume that for all $s =1,2,\dots,N$, 
$\{ (\AA{s}-P)_{ij} : 1 \le i \le j \le n \}$ are independent, mean $0$, 
and for each $s \in [N]$ and $i, j \in [n]$,
 $(\AA{s}-P)_{ij}$ is sub-gamma with parameters $(\nu_{s,ij}, b_{s,ij})$. 
\end{example}

In all of these examples, there are several inference questions we
may wish to ask. In this work, we focus on 
\begin{enumerate}
\item Recovering the matrix $P$; 
\item Recovering $X \in \R^{n \times d}$ when  $P = X X^T$; 
\item Recovering community memberships when $P$ corresponds to a stochastic block model \citep{HolLasLei1983}.
\end{enumerate}
Given that the observed graphs
$\AA{1},\AA{2},\dots,\AA{N}$ differ in their noise structure,
the question arises as to how to combine these graphs to estimate $P$
(or $X$ or the community memberships).  
When the observed graphs are drawn i.i.d.\ from the same
distribution, the sample mean
$\Abar = N^{-1} \sum_{s=1}^N \AA{s}$ is a natural estimate of $P$,
and has been studied in this context in \cite{TanKetVogPriSus2016}
and in a related test statistic in \cite{CheJosLinZhoKol2020}.

However, in our more general setting, where individual networks
and/or edges may be more noisy than others,
 de-emphasizing noisier observations will lead to a more reliable estimate of $P$.
We pursue this by choosing data-dependent weights
$\{ \what_s \ge 0 : s=1,2,\dots,N \}$ with $\sum_{s=1}^N \what_s = 1$
so that the weighted mean estimate $\Ahat = \sum_{s=1}^N \what_s \AA{s}$
is optimal in some sense, or at least provably better than
the sample mean $\Abar$.

\begin{remark}[Positive semi-definite assumption on $P$]
{\em
In what follows,
we make the additional assumption that the expectation $P \in \R^{n \times n}$
is positive semi-definite, 
so that $P = X X^T$ for some $X \in \R^{n \times d}$, where $d \le n$.
This assumption can be removed
using the techniques in \cite{RubPriTanCap2017},
at the cost of added notational complexity.
Thus, for ease of exposition, we confine ourselves to the case where $P = X X^T$, bearing in mind that our results can be easily extended to any $P$.  
}
\end{remark}

\subsection{Recovering Spectral Structure} \label{subsec:lowrank}
Throughout this paper, we will begin from the standard first step in spectral clustering
\citep{RohChaYu2011}.
Following the terminology of \cite{SusTanFisPri2012},
we refer to this as adjacency spectral embedding (ASE). 
Given any adjacency matrix $A \in \R^{n \times n}$,
with a rank $d$ expectation $P$,
write $P = X X^T = \UP \SP \UP^T \in \R^{n \times n}$
where $\SP \in \R^{d \times d}$ is diagonal with entries given by the $d$
non-zero eigenvalues of $P$, and the $d$ corresponding orthonormal
eigenvectors are the columns of $\UP \in \R^{n \times d}$.
Since $P = X X^T = X Q (XQ)^T$ for orthogonal $Q \in \R^{d \times d}$,
the matrix $X$ is only identifiable up to an orthogonal rotation.
Thus, we take $X = \UP \SP^{1/2}$ without loss of generality.
One can view the rows of $X \in \R^{n \times d}$
as latent positions in $\R^d$ associated with the vertices,
with the expectation of an edge between nodes $i$ and $j$  given by the inner product $P_{ij} = X_i^T X_j$ of their latent positions, where $X_i \in \R^d$ is the $i$-th row of $X$.   This view motivates the random dot product graph model 
\citep{RDPGsurvey}, in which the latent positions are first drawn
i.i.d.\ from some underlying distribution on $\R^d$ and
edges are generated independently conditioned on the latent positions.
The natural estimate of the matrix $X = \UP \SP^{1/2}$ is then
\begin{equation*}
 \ASE(A,d) = \UA \SA^{1/2} \in \R^{n \times d},
\end{equation*}
where the eigenvalues in $\SA$ and eigenvectors in $\UA$
now come from $A$ rather than the unknown $P$.  
One can show that under appropriate conditions,
the ASE recovers the matrix $X$ up to an orthogonal rotation
that does not affect the estimate $\Phat = \UA \SA \UA^T \in \R^{n \times n}$.

While we do not focus on the random dot product graph model in this paper,
we make use of several generalizations of results initially
established for that model.
These results are summarized in Appendix~\ref{apx:genericbound}.
In general,
selecting the embedding dimension $d$ (i.e., estimating the rank of $P$)
is an interesting and challenging problem
\citep[see, e.g.,][]{FisSusTanVogPri2013,HanYanFan2019}, but it is not the
focus of the present work, and we assume throughout that $d$ is known.
When $P$ is rank $d$, one can show that under suitable growth conditions,
the gap between the $d$-th largest
eigenvalue and the $(d+1)$-th largest eigenvalue of $P$ grows with $n$,
and the same property holds for the adjacency matrix $A$.
As a result, it is not unreasonable to assume that one can
accurately determine the appropriate dimension $d$ when the number of vertices
$n$ is large.

\section{Methods and Theoretical Results}
\label{sec:method}
Generally speaking, we are interested in
estimators of the form $\Ahat = \sum_{s=1}^N \what_s \AA{s}$,
where $\{ \what_s \}_{s=1}^N$ are data-dependent nonnegative weights
summing to $1$. 
In particular, we study how well the rank-$d$ eigenvalue truncation of
$\Ahat$ approximates the true rank-$d$ expectation $P$.
We measure this either by bounding the difference $ \Ahat - P $ in some matrix norm or by proving that we can successfully recover the matrix
$X \in \R^{n \times d}$ from the estimate $\Ahat$.

In this section, we largely consider the weights $\{ w_s \}_{s=1}^N$ 
to be fixed, rather than data-dependent.
The special case of normally-distributed edge noise is illustrative,
as it suggests a certain choice of weights
in the more general case (see Theorem~\ref{thm:subexpXtilde}).
In Section~\ref{sec:estimate}, we will estimate 
these optimal weights and replace the fixed $\{ w_s \}_{s=1}^N$
with data-dependent estimates $\{ \what_s \}_{s=1}^N$.

\subsection{Normal edges with subject-specific variance}
\label{subsec:gaussian}

Return for a moment to Example~\ref{ex:gaussian}, in which
$\AA{s}_{i,j} \sim \calN(P_{i,j}, \rho_s)$, independent for 
all $1 \le s \le N$ and $1 \le i \le j \le n$.
This very simple setting suggests a choice for
the data-dependent weights $\{ \what_s \}_{s=1}^N$
when constructing the matrix $\Ahat$.
Indeed, a natural extension of the estimator suggested by this example 
will turn out to be the right choice in the more general settings
described in Section~\ref{sec:problemsetup}.
The following proposition follows immediately from
writing out the joint log-likelihood
of the $N$ observed networks and rearranging terms.
\begin{proposition} \label{prop:maxlik}
Suppose $\E \AA{s} = P \in \R^{n \times n}$
for all $s \in [N]$, and $P = X X^T$ for some $X \in \R^{n \times d}$.
If for all $s=1,2,\dots,N$ the edges
$\{ (\AA{s}-P)_{i,j} : 1 \le i \le j \le n \}$
are i.i.d.\ normal mean $0$ and known variance $\rho_s > 0$,
then the maximum likelihood 
estimate for $X \in \R^{n \times d}$ (up to an orthogonal rotation) is given by
$\ASE((\sum_{t=1}^N \rho_t^{-1})^{-1} \sum_{s=1}^N \AA{s}/\rho_s, d)$.
\end{proposition}

Motivated by this proposition, consider the plug-in estimator given by
\begin{equation} \label{eq:def:gauss_plugin}
\Xhat = \ASE\left( \left( \sum_{t=1}^N \rhohat_t^{-1} \right)^{-1}
	\sum_{s=1}^N \rhohat_s^{-1} \AA{s}, d \right),
\end{equation}
where $\rhohat_s$ is an estimate of the variance
$\rho_s$ of the edges in network $s$.
Given $P$, one could naturally use the MLE for $\rho_s$.
Since $P$ is unknown, we estimate it as
$\Phat^{(s)} = \Xhat^{(s)} \Xhat^{(s) T}$, where
$\Xhat^{(s)} = \ASE(\AA{s}, d)$,
and plug in for the MLE of $\rho_s$ to obtain,
for $s=1,2,\dots,N$, 
\begin{equation} \label{eq:def:rhohat}
\rhohat_s
= \sum_{1 \le i \le j \le n} \frac{ 2(\AA{s} - \Phat^{(s)})_{i,j}^2 }{ n(n+1) }.
\end{equation}

With $O(n^2)$ edges in each network, the estimates $\{ \rhohat_s \}_{s=1}^N$
converge to the true variances $\{ \rho_s \}_{s=1}^N$
in such a way that the estimation rate of the plug-in estimator
in Equation~\eqref{eq:def:gauss_plugin}
matches that of the maximum-likelihood
estimator in Proposition~\ref{prop:maxlik}.
The following proposition makes this claim precise.

\begin{proposition} \label{prop:gauss:XhatXopt}
Let $\AA{1},\AA{2},\dots,\AA{N}$ be independent adjacency matrices
with common expectation $\E \AA{s} = P = X X^T$, where $X \in \R^{n \times d}$,
and suppose that for each $s \in [N]$,
$\{ (\AA{s}-P)_{i,j} : 1 \le i \le j \le n \}$ 
are independent $\calN(0, \rho_s)$.
Let $\Xopt \in \R^{n \times d}$ denote the maximum likelihood estimator
under the assumption of known variances, 
as described in Proposition~\ref{prop:maxlik}.

Suppose that
\begin{equation} \label{eq:variancesum}
\sum_{s=1}^N \rho_s^{-1}
= \omega\left( \frac{n \log^2 n}{\lambda_d^2(P)} \right),
\end{equation}
Then for all suitably large $n$, there exists orthogonal matrix
$\Vopt \in \R^{d \times d}$ such that
\begin{equation*}
\| \Xopt - X \Vopt \|_{\tti}
\le
\frac{ Cd }{ \lambda_d^{1/2}(P) }
        \left( \sum_{s=1}^N \rho_s^{-1} \right)^{-1/2}
        + \frac{ Cd \kappa(P) n }{ \lambda_d^{3/2}(P) }
        \left( \sum_{s=1}^N \rho_s^{-1} \right)^{-1}.
\end{equation*}
Further, let $\Xhat \in \R^{n \times d}$ denote the estimator defined in
Equation~\eqref{eq:def:gauss_plugin}.
For all suitably large $n$, there exists orthogonal matrix
$V \in \R^{d \times d}$ such that with probability $1-Cn^{-2}$,
\begin{equation*}
\| \Xhat - X V \|_{\tti}
\le \frac{ Cd }{ \lambda_d^{1/2}(P) }
        \left( \sum_{s=1}^N \rho_s^{-1} \right)^{-1/2}
        + \frac{ Cd \kappa(P) n }{ \lambda_d^{3/2}(P) }
        \left( \sum_{s=1}^N \rho_s^{-1} \right)^{-1}.
\end{equation*}
\end{proposition}
This proposition is a special case of
Theorem~\ref{thm:subgamma:XhatXopt} in Section~\ref{sec:estimate},
and thus we delay its proof until then.

\subsection{Sub-gamma edges}
\label{subsec:subgamma}

In settings like those in
Examples~\ref{ex:subgammasubject} and~\ref{ex:subgamma},
where there are no longer parameters controlling the noise distribution,
we must resort to more general concentration inequalities.
Our main tool in this setting is a generalization
of a bound on the error in recovering
$X = \UP \SP^{1/2} \in \R^{n \times d}$.
Given a single adjacency matrix $A \in \R^{n \times n}$
with $\E A = P = X X^T \in \R^{n \times n}$,
a natural estimate of $X \in \R^{n \times d}$ is $\ASE(A,d)$.
The following lemma bounds the difference between
this estimate and an orthogonal rotation ($V$ in the result below)
of $X = \UP \SP^{1/2}$.
This bound makes no use of a particular error structure,
but instead bounds the difference in terms of
three different norms of $A-P$.
This error term can then be bounded
using standard concentration inequalities,
which we will do below in the proof of Theorem~\ref{thm:subexpXtilde}.

\begin{lemma} \label{lem:genericbound}
Let $P = X X^T = \UP \SP \UP^T \in \R^{n \times n}$
be a rank-$d$ matrix with non-zero eigenvalues
$\lambda_1(P) \ge \lambda_2(P) \ge \cdots \ge \lambda_d(P) > 0$.
Let $A \in \R^{n \times n}$ be a random symmetric matrix
for which there exists a constant $c_0 \in [0,1)$
such that with probability $p_0$,
\begin{equation} \label{eq:assum:specbound}
\| A - P \| < c_0 \lambda_d(P)
\end{equation}
for all suitably large $n$.
Letting $\Xhat = \ASE( A, d )$,
for all suitably large $n$,
there exists a random orthogonal matrix $V = V_n \in \R^{d \times d}$
such that with probability $p_0$
\begin{equation*} \begin{aligned}
\| \Xhat &- X V \|_{\tti} \\
&\le
  \frac{ \| (A-P)\UP \|_{\tti} }{ \lambda_d^{1/2}(P) }
  + \frac{ C \| \UP^T(A-P)\UP \|_F }{ \lambda_d^{1/2}(P) }
   + \frac{ C d \| A-P \|^2 \kappa(P) }{ \lambda_d^{3/2}(P) } .
  \end{aligned} \end{equation*}
\end{lemma}

This lemma generalizes
Theorem 18 in \cite{LyzTanAthParPri2017} and
Lemma 1 in \cite{LevAthTanLyzPri2017}.
Details are included in Section~\ref{apx:genericbound}
of the Appendix.
We note that recent work
\citep[][Theorem 4.2]{CapTanPri2019}
established a $(\tti)$-norm bound for eigenvector recovery,
a problem related to, but fundamentally different from,
the problem of recovering $X \in \R^{n \times d}$ considered in
Lemma~\ref{lem:genericbound}.

In order to apply Lemma~\ref{lem:genericbound} to the
random matrix $\Atilde = \sum_{s=1}^N w_s \AA{s}$,
we need to ensure that the spectral condition in
Equation~\eqref{eq:assum:specbound} holds.
Toward that end,
the following lemma bounds the spectral error between
$\Atilde$ and $P$ in terms of the weights and the sub-gamma parameters.
We begin by considering the case where the weights
$\{ w_s \}_{s=1}^N$ are {\em not} data dependent.
In Section~\ref{sec:estimate}, we will consider the case where
the weights are a function of the networks $\AA{1},\AA{2},\dots,\AA{N}$.

\begin{lemma} \label{lem:mxspecbound}
Suppose that $\AA{1},\AA{2},\dots,\AA{N}$ are independent symmetric adjacency
matrices with shared expectation $\E \AA{s} = P \in \R^{n \times n}$,
and suppose that 
$\{ (\AA{s}-P)_{i,j} : s \in [N], 1 \le i \le j \le n\}$
are independent sub-gamma random variables
with parameters $(\nu_{s,i,j}, b_{s,i,j})$.
Let $w_1,w_2,\dots,w_N \ge 0$ be non-random weights
with $\sum_{s=1}^N w_s = 1$.
Then with probability at least $1 - Cn^{-2}$,
\begin{equation*}
\left\| \Atilde - P \right\| \le \frac{ 15 \sqrt{2\mxbernsymbol^2} }{ 2 }\log n ,
\end{equation*}
where
\begin{equation*}
\mxbernsymbol^2
= 2 \max_{i\in[n]} \sum_{s=1}^N \sum_{j=1}^n
        w_s^2 (\sqrt{2 \nu_{s,i,j}} + 2b_{s,i,j})^2.
\end{equation*}
\end{lemma}
Lemma~\ref{lem:mxspecbound} follows from a standard matrix Bernstein inequality
\citep{Tropp2012}. Details are provided in Appendix~\ref{apx:subexpXtilde}.

While bounds for recovering $X$ are also possible under the setting of
Example~\ref{ex:subgamma},
in which $(\AA{s} - P)_{i,j}$ is $(\nu_{s,i,j},b_{s,i,j})$-sub-gamma
for each $s \in [N], i,j \in [n]$,
the bounds are comparatively complicated functions of these parameters.
For simplicity, we state the following theorem for the case where
the edges in the $s$-th network are independent $(\nu_s,b_s)$-sub-gamma
random variables.

\begin{theorem} \label{thm:subexpXtilde} 
Under the setting of Lemma~\ref{lem:mxspecbound}, with the additional
condition that $P = X X^T$ for some $X \in \R^{n \times d}$
and $(\nu_{s,i,j},b_{s,i,j}) = (\nu_s,b_s)$ for all $i,j \in [n]$, let
$\Xtilde = \ASE(\Atilde,d)$.
Suppose that the weights $\{w_s\}_{s=1}^N$ and sub-gamma parameters
$\{ (\nu_s,b_s) \}_{s=1}^N$ are such that
\begin{equation} \label{eq:paragrowth}
\sum_{s=1}^N w_s^2(\nu_s + b_s^2)
= o\left( \frac{\lambda_d^2(P)}{n \log^2 n} \right)
\end{equation}
Then with probability $1-Cn^{-2}$
there exists an orthogonal  matrix $V \in \R^{d \times d}$ such that
\begin{equation*} \label{eq:betadef}
\begin{aligned}
\| \Xtilde &- X V \|_{\tti} \\
&\le \frac{ Cd }{ \lambda_d^{1/2}(P) }
        \left( \sum_{s=1}^N w_s^2(\nu_s + b_s^2) \right)^{1/2} \log n
+ \frac{ Cd n \kappa(P) }{ \lambda_d^{3/2}(P) }
	\left( \sum_{s=1}^N w_s^2(\nu_s + b_s^2) \right) \log^2 n.
\end{aligned}
\end{equation*}
\end{theorem}
This theorem follows from applying Lemma~\ref{lem:genericbound}
with $A = \Atilde$, using Lemma~\ref{lem:mxspecbound}
and Equation~\eqref{eq:paragrowth} to ensure that
Equation~\eqref{eq:assum:specbound} holds,
and applying standard concentration inequalities to control the resulting
bound on $\| \Xtilde - X V \|_{\tti}$.
Details can be found in Appendix~\ref{apx:subexpXtilde}.

\begin{remark} \label{rem:weights} {\em
To illustrate the consequences of Theorem~\ref{thm:subexpXtilde},
let us consider two different settings for the sub-gamma parameters.
First, consider the case when all
$\{ \nu_s + b_s^2 \}_{s=1}^N$ are of constant order.
Setting $w_s = N^{-1}$ for $s=1,2,\dots,N$ then yields
$\sum_s w_s^2(\nu_s + b_s^2) = O( N^{-1} )$,
and Theorem~\ref{thm:subexpXtilde}
shows that, unsurprisingly, network averaging
yields faster recovery in estimating $X$ compared with the single-network
setting, at least in the case where all of the networks have
comparable levels of edge uncertainty.

Contrast this case with the setting where
$(\nu_1+b_1^2) = N^2$
and $(\nu_s + b_s^2) = 1$ for $s=2,3,\dots,N$.
Applying Theorem~\ref{thm:subexpXtilde},
if we take $w_s = N^{-1}$ for all $s \in [N]$,
our estimation accuracy for $X$ is controlled by
\begin{equation} \label{eq:ratecompare:const}
\sum_s w_s^2(\nu_s + b_s^2)
= \frac{1}{N^2}\left( N^2 + N - 1\right)
= O(1).
\end{equation}
On the other hand, setting the weights to 
\begin{equation*}
w_1 = \frac{ N^{-2} }{ N^{-2} + N-1 }
~~~\text{ and }~~~
w_s = \frac{ 1 }{ N^{-2} + N-1 }~~~\text{ for } s=2,3,\dots,N,
\end{equation*}
implies that the estimation rate in Theorem~\ref{thm:subexpXtilde} is controlled by
\begin{equation*} 
w_1^2(\nu_1 + b_1^2) + \sum_{s=2}^N w_s^2(\nu_s+b_s^2)
= \frac{ 1 }{ N^{-2} + N - 1 } = O(N^{-1}).
\end{equation*}
Comparing this rate with the rate in Equation~\eqref{eq:ratecompare:const},
we see that weighted network averaging can yield qualitatively better
estimation in the setting where one or more networks have much higher
uncertainty in their edge measurements.
}
\end{remark}

\subsection{Selecting network weights} 
\label{subsec:weightselect}

When $\AA{s}$ has sub-gamma edge noise
with parameters $(\nu_s,b_s)$ common for all edges,
the bound in Theorem~\ref{thm:subexpXtilde} is a monotone function
of the quantity $\sum_s w_s^2(\nu_s + b_s^2)$.
This suggests that we choose
the weights $\{ w_s \}_{s=1}^N$ so as to minimize this quantity,
which is achieved by taking
\begin{equation} \label{eq:def:subgamma:wopt}
w_s = \wopt_s =
 \frac{ (\nu_s + b_s^2)^{-1} }{ \sum_{t=1}^N (\nu_t + b_t^2)^{-1} }
\end{equation}
for each $s \in [N]$.
Applying Lemma~\ref{lem:mxspecbound} to $\Aopt = \sum_s \wopt_s \AA{s}$,
and using the assumption that $(\nu_{s,i,j}, b_{s,i,j}) = (\nu_s,b_s)$
for all $s \in [N]$ and $i,j \in [n]$,
we conclude that with high probability,
\begin{equation} \label{eq:Aopt:rate}
\left\| \Aopt - P \right\| 
\le C \sqrt{ \frac{ n }{ \sum_s \nub{s}^{-1} } } \log n,
\end{equation}
where we have used the fact that
\begin{equation*}
\nu_s + b_s^2 \le (\sqrt{\nu_s} + b_s)^2 \le 2(\nu_s + b_s^2).
\end{equation*}

Perhaps surprisingly, when the edges are normally distributed about their
expectations, this selection of weights $\{ \wopt_s \}_{s=1}^N$
yields the minimax optimal rate for recovering $P$ in spectral norm.
When $(\AA{s}-P)_{i,j} \sim \calN(0,\rho_s)$ for all $1 \le i \le j \le n$,
we can take the sub-gamma parameters to be $(\nu_s,b_s) = (\rho_s,0)$
for all $s \in [N]$, and the bound in Equation~\eqref{eq:Aopt:rate} becomes
\begin{equation*} 
\left\| \Aopt - P \right\|
\le C \sqrt{ \frac{ n }{ \sum_s \rho_s^{-1} } } \log n,
\end{equation*}
and this matches the minimax rate, as the following result shows.
A proof can be found in the Appendix.

\begin{theorem} \label{thm:minimax}
Let $\AA{1},\AA{2},\dots,\AA{N} \in \R^{n \times n}$ be independent
symmetric adjacency matrices, with
$\{ (\AA{s}_{ij} - P_{ij}) : 1 \le i \le j \le n \}$ drawn
i.i.d.\ from a normal with mean $0$ and variance $\rho_s$
for each $s = 1,2,\dots,N$.
Then, letting $\symmetricmxs_n = \{ P \in \R^{n \times n} : P = P^T\}$,
for all $n \ge 2$, 
\begin{equation*}
\inf_{\Phat} \sup_{P \in \symmetricmxs_n}
        \E \| \Phat - P \|
\ge C \sqrt{n} \left( \sum_{s=1}^N \rho_s^{-1} \right)^{-1/2},
\end{equation*}
where the infimum is over all estimators $\Phat$ of $P$.
\end{theorem}

\section{Estimating the Sub-gamma Parameters}
\label{sec:estimate}
In the setting where each network $\AA{s}$ has edges with shared sub-gamma
parameter $(\nu_s,b_s)$,
the results in Sections~\ref{subsec:subgamma} and~\ref{subsec:weightselect}
suggested choosing our network weights according to
Equation~\eqref{eq:def:subgamma:wopt}.
Of course, in practice, we do not know the sub-gamma parameters
$\{(\nu_s,b_s)\}_{s=1}^N$
and hence we must estimate them in order to obtain estimates of the
optimal weights. 
Since $(\nu_s + b_s^2)$ is (up to a constant factor) an upper bound on the
variances of the $\{ (\AA{s} - P)_{ij} :1 \le i \le j \le n \}$,
a natural estimate of $\wopt_s$ is
\begin{equation} \label{eq:def:subgamma:what}
\what_s = \frac{ \rhohat_s^{-1} }{ \sum_{t=1}^N \rhohat_t^{-1} },
\end{equation}
where, letting $\Phat^{(s)} \in \R^{n \times n}$
be the rank-$d$ truncation of $\AA{s}$ for $s=1,2,\dots,N$,
\begin{equation} \label{eq:def:subgamma:rhohat}
\rhohat_s
= \frac{ \sum_{1 \le i \le j \le n}
	\left(\AA{s}_{ij} - \Phat^{(s)}_{ij} \right)^2 }
	{16 n(n+1) }.
\end{equation}
Comparison with Equation~\eqref{eq:def:rhohat} reveals that this is,
in essence, the same estimation procedure that we derived in
Section~\ref{subsec:gaussian}, extended to the case of sub-gamma edges.
The factor of 16 in the denominator comes from replacing the
equality $\E(\AA{s} - P)_{ij}^2 = \rho_s$
with the sub-gamma moment bound \citep[][Chapter 2, Theorem 2.3]{BLM}
\begin{equation} \label{eq:BLM:varbound}
\E(\AA{s} - P)_{ij}^2 \le 8\nu_s + 32b_s^2 \le 32(\nu_s + b_s^2).
\end{equation}
Just as in Section~\ref{subsec:gaussian},
the estimated weights $\{ \what_s \}_{s=1}^N$ are such that
the plug-in estimate $\ASE( \sum_s \what_s \AA{s}, d)$ 
recovers the true matrix $X \in \R^{n \times d}$
(up to orthogonal nonidentifiability) at the same rate as we
would obtain if we knew the true sub-gamma parameters.
\begin{theorem} \label{thm:subgamma:XhatXopt}
Suppose that $\AA{1},\AA{2},\dots,\AA{N}$ are independent symmetric adjacency
matrices with shared expectation $\E \AA{s} = P = X X^T \in \R^{n \times n}$,
where $X \in \R^{n \times d}$.
Suppose further that for each $s \in [N]$,
$\{ (\AA{s}-P)_{i,j} : 1 \le i \le j \le n\}$
are independent sub-gamma random variables
with parameters $(\nu_s,b_s)$.
Let $\{ \wopt_s \}_{s=1}^N$ and $\{ \what_s \}_{s=1}^N$
be the weights defined in Equations~\eqref{eq:def:subgamma:wopt}
and~\eqref{eq:def:subgamma:what}, respectively,
and define the estimators
\begin{equation*}
\Xhat = \ASE\left( \sum_{s=1}^N \what_s \AA{s}, d \right),
~~~
\Xopt = \ASE\left( \sum_{s=1}^N \wopt_s \AA{s}, d \right).
\end{equation*}
Suppose that the parameters $n,d$ and $N$ grow in such a way
that $d/n \le 1$ for all suitably large $n$,
and, some positive integer $k$,
\begin{equation} \label{eq:updatedgrowth}
 \frac{ n^{k-2} d^k \left( \log N + \log n \right)^{4k} }{N} = \Omega( 1 ).
\end{equation}
Suppose further that the sub-gamma parameters
$\{ (\nu_s,b_s) \}_{s=1}^N$ are such that
\begin{equation} \label{eq:subgammagrowth}
\left( \frac{1}{N} \sum_{s=1}^N (\nu_s + b_s^2)^{-1} \right)^{-1}
= o\left( \frac{ N \lambda_d^2(P) }{ n \log^2 n} \right).
\end{equation}
For each $s \in [N]$, define
\begin{equation} \label{eq:def:tau}
\tau_s = \frac{ \sum_{1 \le i \le j \le n} \E (\AA{s} - P)_{ij}^2 }
		{ 16n(n+1) },
\end{equation}
and suppose that
\begin{equation} \label{eq:assum:ratiogrowth}
  \lim_{n \rightarrow \infty}
	\frac{ \sqrt{d} (\log N + \log n)^2
		\max_{s \in [N]} \tau_s^{-1} (\nu_s + b_s^2) }
		{ \sqrt{n} }
	= 0
\end{equation}
in such a way that
\begin{equation} \label{eq:assum:harmonicish}
 \left(\max_{s \in [N]} \tau_s^{-1} \nub{s} \right)
        \sqrt{ \sum_{s=1}^N \frac{ \tau_s^{-1} \sum_{t=1}^N \nub{t}^{-1}  }
                { \nub{s}^{-1} \sum_{t=1}^N \tau_t^{-1}  } }
= O\left( \frac{ \sqrt{n} \log n }{ \sqrt{d}\left(\log n + \log N\right)^3 } \right).
\end{equation}
Provided that
\begin{equation} \label{eq:assum:harmonicish2}
\sum_{s=1}^N \frac{ \nub{s}^{-1} }{ \sum_t \nub{t}^{-1} }
		\left(1 - \frac{\tau_s^{-1} }{\wopt_s \sum_t \tau_t^{-1} }
			\right)^2
= O\left( \frac{ \log n }{ \sqrt{N}\left( \log n + \log N \right) } \right),
\end{equation}
then for all suitably large $n$, it holds with probability $1-O(n^{-2})$
that there exist orthogonal matrices
$V,\Vopt \in \R^{d \times d}$ such that
\begin{equation*} \begin{aligned}
\| \Xhat - X V \|_{\tti} &\le
\frac{ Cd }{ \lambda_d^{1/2}(P) }
        \left( \sum_{s=1}^N (\nu_s + b_s^2)^{-1} \right)^{-1/2}
        + \frac{ Cd \kappa(P) n}{ \lambda_d^{3/2}(P) }
        \left( \sum_{s=1}^N (\nu_s + b_s^2)^{-1} \right)^{-1} \\
&~~~\text{and} \\
\| \Xopt - X \Vopt \|_{\tti} &\le
	\frac{ Cd }{ \lambda_d^{1/2}(P) }
        \left( \sum_{s=1}^N (\nu_s + b_s^2)^{-1} \right)^{-1/2}
	+ \frac{ Cd \kappa(P) n }{ \lambda_d^{3/2}(P) }
        \left( \sum_{s=1}^N (\nu_s + b_s^2)^{-1} \right)^{-1}.
\end{aligned} \end{equation*}
That is, the plug-in estimator $\Xhat$, based on the estimates weights
$\{ \what_s \}_{s=1}^N$,
recovers $X \in \R^{n \times d}$
at the same rate as the estimator
$\Xopt$ based on the optimal weights $\{ \wopt_s \}_{s=1}^N$.
\end{theorem}

The proof is given in Appendix~\ref{apx:estimation}.
Note that the bound on
$\| \Xopt - X \Vopt \|_{\tti}$ follows straightforwardly
from Theorem~\ref{thm:subexpXtilde}.
The analysis of $\sum_s \what_s \AA{s}$ requires more care,
since the weights $\{ \what_s \}_{s=1}^N$ now depend on the observed networks.

\begin{remark} {\em
The quantities $\{ \tau_s^{-1} \nub{s} \}_{s=1}^N$
in Equations~\eqref{eq:assum:ratiogrowth},~\eqref{eq:assum:harmonicish}
and~\eqref{eq:assum:harmonicish2} are, in essence,
measures of the tightness of the sub-gamma tail
bounds $\E (\AA{s} - P)_{i,j}^2 \le 32(\nu_s + b_s^2)$.
For example, the quantity controlled by Equation~\eqref{eq:assum:harmonicish2}
is the $\chi^2$  divergence between the distribution on $[N]$
encoded by the optimal weights
$\{ \wopt_s : s \in [N] \}$ and the distribution given by
$u_s = \tau_s^{-1}/\sum_t \tau_t^{-1}$.
In the simplest case, when $\{ (\AA{s}-P)_{i,j} : 1 \le i \le j \le n \}$
are i.i.d.\ $\calN(0,\rho_s)$ for some $\rho_s > 0$,
we have $(\nu_s,b_s) = (\rho_s,0)$,
so that $32\tau_s = \rho_s = \nub{s}$
and thus $\tau_s^{-1} \nub{s} = 32$ for all $s \in [N]$.
The growth conditions in
Equations~\eqref{eq:assum:ratiogrowth},~\eqref{eq:assum:harmonicish}
and~\eqref{eq:assum:harmonicish2} are then satisfied trivially,
so long as
$d (\log N + \log n)^4 = o(n)$.
}
\end{remark}

\begin{remark} {\em
As mentioned just before Theorem~\ref{thm:subexpXtilde},
the more general case, in which
$(\AA{s} - P)_{i,j}$ is $(\nu_{s,i,j},b_{s,i,j})$-sub-gamma
for each $s \in [N], i,j \in [n]$,
is notably more complicated to analyze than the setting where
there is a single $(\nu_s,b_s)$ parameter to estimate for each network
$s=1,2,\dots,N$.
We expect that mild structural assumptions on the
matrices $[\nu_{s,i,j}]_{i,j=1}^n$ and
$[b_{s,i,j}]_{i,j=1}^n$ (e.g., low rank)
would yield similar estimation procedures to that described above.
The technical results presented in Appendix~\ref{apx:subexpXtilde}
give an indication of the cumbersome notation required to
handle more general structural assumptions on the sub-gamma parameters.
We leave this generalization for future work.  
}
\end{remark}

\section{Perfect Clustering with Sub-gamma Edges}
\label{sec:perfect}
With Theorem~\ref{thm:subexpXtilde} in hand, we obtain an immediate
bound on the community misclassification rate in block models
by an argument similar to that in~\cite{LyzSusTanAthPri2014}.

This bound holds for {\em weighted} stochastic blockmodels (SBMs),
in which multiple weighted graphs are drawn with a shared block structure.
Weighted versions of the stochastic blockmodel have received increasing
attention in recent years \citep[see, e.g.,][]{AicJacCla2015}.
Extensions to the case of weighted edges in a multiple-network setting
was recently discussed in \cite{KhiLoh2018},
where the authors considered a network time series problem.
The definition given here subsumes any variant on the weighted SBM in which,
conditional on the community assignments,
edge distributions are independent and obey a sub-gamma tail bound.


\begin{definition}[Joint Sub-gamma SBM] {\em
Let $B \in [0,1]^{K \times K}$ 
and define the community membership matrix $Z \in \{0,1\}^{n \times K}$
by $Z_{ik} = 1$ if the $i$-th vertex belongs to community $k$
and $Z_{ik} = 0$ otherwise.
We say that random adjacency matrices
$\AA{1},\AA{2},\dots,\AA{N} \in \R^{n \times n}$ are
{\em jointly sub-gamma stochastic block model}, written
\begin{equation*}
(\AA{1},\AA{2},\dots,\AA{N}) \sim \JsubSBM(n,B,Z,\{(\nu_s,b_s)\}_{s=1}^N),
\end{equation*}
if conditional on $Z$, the $N$ adjacency matrices are independent
with a common expectation $\E \AA{s} = Z B Z^T$ ($s=1,2,\dots,N$),
and within each adjacency matrix $\AA{s}$,
$\{ \AA{s}_{ij} : 1 \le i \le j \le n \}$ are independent
$(\nu_s,b_s)$-sub-gamma random variables.
} \end{definition}

Our theoretical results from Section~\ref{sec:method}
have an immediate implication for
detection and estimation of shared community structure
in the joint sub-gamma SBM model.  
This result generalizes Theorem 6 in \cite{LyzSusTanAthPri2014}.

\begin{theorem} \label{thm:perfect}
Suppose that 
$(\AA{1},\dots,\AA{N})$ are drawn from $\JsubSBM(n,B,Z, \{(\nu_s,b_s)\}_{s=1}^N),$
where $B = Y Y^T \in \R^{K \times K}$ is fixed with
$Y \in \R^{K \times d}$ having $K$ distinct rows given by
$Y_1,Y_2,\dots,Y_K \in \R^d$.
Let $\Atilde = \sum_{s=1}^N w_s \AA{s}$, where
$\{w_s\}_{s=1}^N$ are fixed non-negative weights summing to $1$.
Let $\nmin = \min_{k \in [K]} \sum_{i=1}^n Z_{ik}$
denote the size of the smallest community.
Suppose that $\{ (w_s,\nu_s,b_s) \}_{s=1}^N$ obey the growth conditions
in Equation~\eqref{eq:paragrowth} and that the
smallest community grows as
\begin{equation} \label{eq:Nmingrowth}
\nmin = \omega\left(
                d^2 \left( \sum_{s=1}^N w_s^2(\nu_s + b_s^2) \right)
	+ d^2 \left( \sum_{s=1}^N w_s^2(\nu_s + b_s^2) \right)^2\log^4 n \right).
\end{equation}
Let $\tau : [n] \rightarrow [K]$ be the true underlying
assignment of vertices to communities,
so that $\tau(i) = k$ if and only if $Z_{ik} = 1$,
and let $\tauhat : [n] \rightarrow [K]$
be the estimated community assignment function based on
an optimal $K$-means clustering of the rows of $\Xtilde = \ASE( \Atilde, d)$.
Then the communities are recovered exactly almost surely, i.e.,
as $n \rightarrow \infty$,
\begin{equation} \label{eq:perfect}
\Pr\left[ \min_{\pi \in S_K} \left|\{i \in [n] : \pi(\tau(i)) \neq \tauhat(i)\}\right|  \rightarrow 0 \right] = 1
\end{equation}
\end{theorem}
\begin{proof}
The result follows from Theorem~\ref{thm:subexpXtilde}
and the fact that under the stochastic block model with fixed parameters,
we have $\lambda_d(P) = \Omega(n)$
\citep[see, for example, Observation 2 in][]{LevAthTanLyzPri2017}.
The proof is otherwise a direct adaptation of the proof of Theorem 6 in
\cite{LyzSusTanAthPri2014}, and we omit the details.
\end{proof}

%

\begin{remark}[Extensions of Theorem~\ref{thm:perfect}] {\em 
This result can be generalized in two natural directions.
The first would be to allow for the communication matrix $B$ to depend on $n$.
Generally speaking, provided the entries of $B$
do not go to zero too quickly, the quantities $\nmin$ and $\lambda_d(P)$ will 
grow quickly enough to ensure that $\| \Xtilde - X W \|_F \rightarrow 0$,
and Theorem~\ref{thm:perfect} still holds
This extension is straightforward  and we omit the details
(though see Remark~\ref{rem:sparsity} below).
Another generalization
would be to expand the class of clustering algorithms for
which the perfect recovery condition in~\eqref{eq:perfect} holds.
For example, an argument similar to that sketched in
Theorem~\ref{thm:perfect} would apply equally well to another distance-based
clustering method, such as $K$-medians, $K$-medoids
or $K$-centers \citep{GarNeeRao1977}.
When there are $K$ clusters, the matrix $Y$ has $K$ distinct rows,
say, $y_1,y_2,\dots,Y_K \in \R^d$, so that for all $i \in [n]$,
$X_i \in \{y_1,y_2,\dots,y_K\}$.
By Theorem~\ref{thm:subexpXtilde},
provided the parameters grow at suitable rates, for all suitably large $n$,
the rows of $\Xtilde$ lie in $K$ disjoint balls centered at the
$K$ points $\{y_1,y_2,\dots,y_K\}$,
and a clustering solution that does not place a centroid in each of these
$K$ balls can be improved upon by a solution that does.
We leave it for future work to characterize the clustering algorithms that
obtain this recovery guarantee and the precise growth conditions
on the model parameters required for these different algorithms to succeed.
} \end{remark}

\begin{remark}[Incorporating Sparsity] \label{rem:sparsity} {\em
A standard way to incorporate sparsity in the SBM
is to let $B = q_n Y Y^T$, where $q_n \rightarrow 0$ is a sparsity parameter.  
Similar to in our previous Remark,
recovery of the community memberships requires, in essence,
that the rows of $\sqrt{q_n} Y$ are suitably well separated.
This imposes a lower-bound on how quickly the sparsity parameter
$q_n$ can converge to zero.
Specifically, the condition in Equation~\eqref{eq:paragrowth} is satisfied
so long as
\begin{equation*}
\sum_s w_s^2(\nu_s + b_s^2)
= o\left( \frac{ q_n^2 n }{ \log^2 n } \right).
\end{equation*}
Since a Bernoulli with success probability $q$ is a $(q,1)$-sub-gamma random
variable, when $N=1$ this becomes
$(q_n+1)\log^2 n = o( q_n^2 n )$, i.e., $q_n = \omega( n^{-1/2} \log n)$.  For a single network, this 
 is a stricter requirement on the average degree  than the more typical 
$q_n = \omega(n^{-1}) \log^c n)$ for some constant $c \ge 0$.
While the extension to sub-gamma edge distributions
allows a much more general class of noise models,
our general bounds are not necessarily tight
when applied to the highly-structured setting of Bernoulli edges,
where the variance is constrained by the mean.    When this additional structure is present, it is, unsurprisingly, more efficient to leverage it 
\citep[see, e.g.,][]{LeLevLev2018}.
}
\end{remark}

\section{Numerical experiments}
\label{sec:experiments}
We now turn to an experimental investigation of the effect of weighted
averaging. We begin with simulated data,
and then turn to a neuroimaging application.

\subsection{Effect of weighted averaging on estimation}
We begin by investigating the extent to which weighted averaging improves upon
its unweighted counterpart in the case where
we observe multiple weighted graphs with network-specific
edge variances, as in Examples~\ref{ex:gaussian} and~\ref{ex:subgammasubject}.

We consider the following simulation setup.
On each trial,
we generate the rows of $X \in \R^{n \times d}$
independently and identically distributed as
$\calN\left( (1,1,1)^T, \Sigma \right)$, where
\begin{equation*}
\Sigma = \begin{bmatrix} 3 & 2 & 1 \\
			 2 & 3 & 2 \\
			 1 & 2 & 3 \end{bmatrix},
\end{equation*}
and take $P = X X^T$. 
Next, independently for each network $s=1,2,\dots,N$,
we draw edge weights $\{ (\AA{s}-P)_{ij} : 1 \le i \le j \le n \}$
independently from a $0$-mean Laplace distribution
with variance $\sigma^2_s > 0$.
We chose this distribution because it has
heavier tails than the Gaussian while still belonging to the class of
sub-gamma random variables. Similar results to those presented here were also
observed under Gaussian, exponential, and gamma error distributions.
Without loss of generality, we take the first network to have edge variance $\sigma^2_1 \ge 1,$
while all other networks have unit edge variance,
so that $\sigma^2_s = 1$ for $s > 1$.
Thus, the first network is an outlier with higher edge-level variance than
the other observed networks.
We compare weighted and unweighted averaging, with weights estimated as described in Section~\ref{sec:estimate},
to obtain the weighted average $\Ahat = \sum_{s=1}^N \what_s \AA{s}$,
and the unweighted average $\Abar = N^{-1} \sum_{s=1}^N \AA{s}$.
Rank-$d$ eigenvalue truncations of each of these yield 
estimates $\Phatwtd$ and $\Phatunif$, respectively.
We evaluate the weighted estimate by its relative improvement,
\begin{equation*}
\frac{ \| \Phatunif - P \| - \| \Phatwtd - P \| }{ \| \Phatunif - P \| }.  
\end{equation*}
We can think of this quantity as a measure of the outlier's influence. 

We repeat this experiment for different values of the number of vertices $n$,
the number of networks $N$ and the outlier variance $\sigma_1^2$, and average over 50 replications for each setting.  
Figure~\ref{fig:laplacian} summarizes the results for relative improvement measured in Frobenius norm.   We also computed the relative improvement in spectral and  $(2,\infty)$ matrix norms, with similar results (omitted).   
Figure~\ref{subfig:laplacian:var} shows relative improvement as a function
of the outlier variance $\sigma_1^2$, for different values of $N$, with $n=200$ fixed.  
Figure~\ref{subfig:laplacian:nvx} shows relative improvement as a function
of $n$ while holding the outlier variance $\sigma_1^2$
fixed, again for different values of $N$.   Figure~\ref{fig:laplacian} suggests two main conclusions.   First, the relative improvement is never negative, showing that even when the outlier variance is small, there is no disadvantage to using the weighted average. Second, even a single outlier with larger edge variance can significantly impact the unweighted average.   
Similar trends to those seen in Figure~\ref{fig:laplacian} apply to the error in recovering $X$.  

\begin{figure*}[t!]
  \centering
  \subfloat[\label{subfig:laplacian:var}]{ \includegraphics[width=0.45\textwidth]{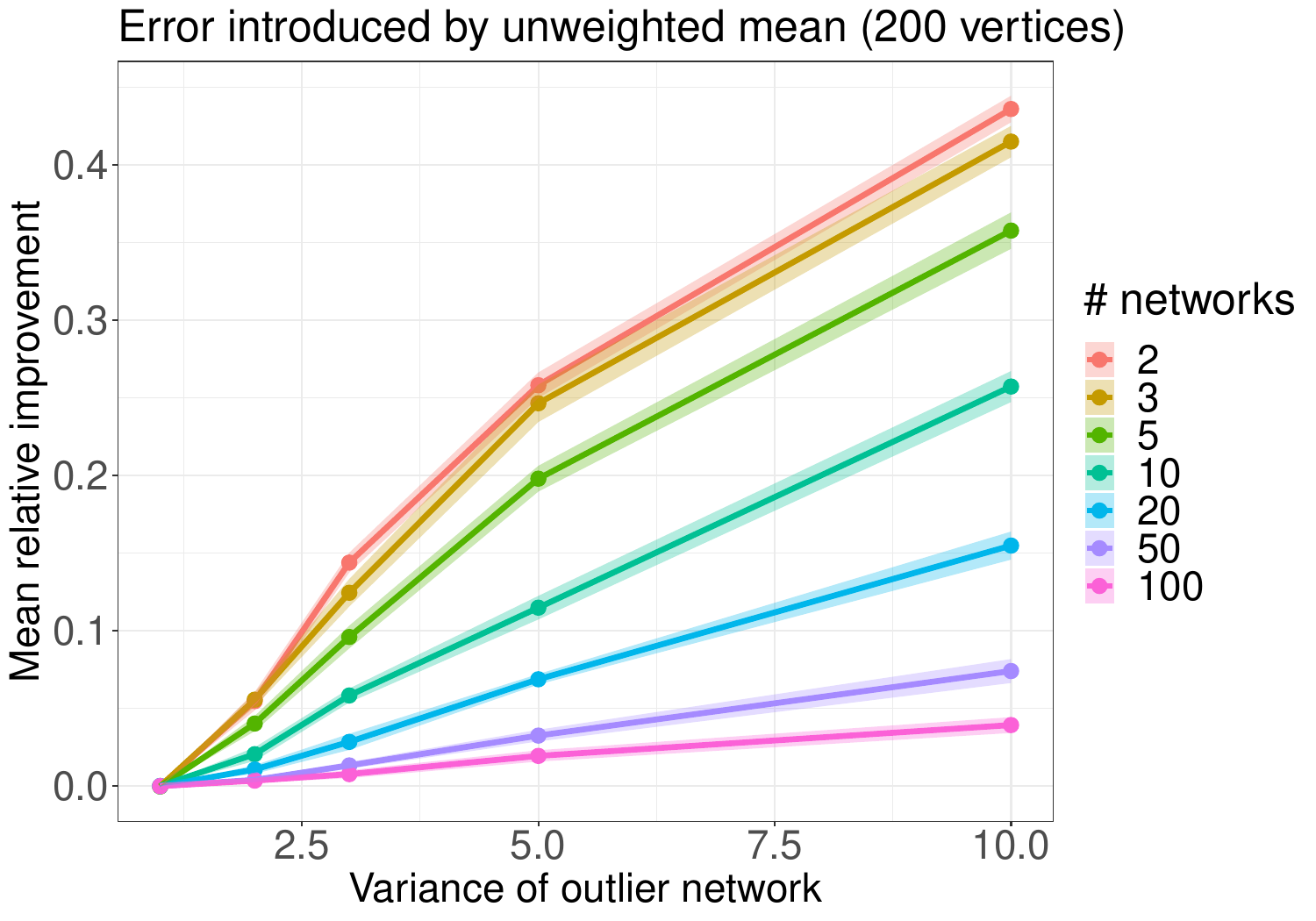} }
  \subfloat[\label{subfig:laplacian:nvx}]{ \includegraphics[width=0.45\textwidth]{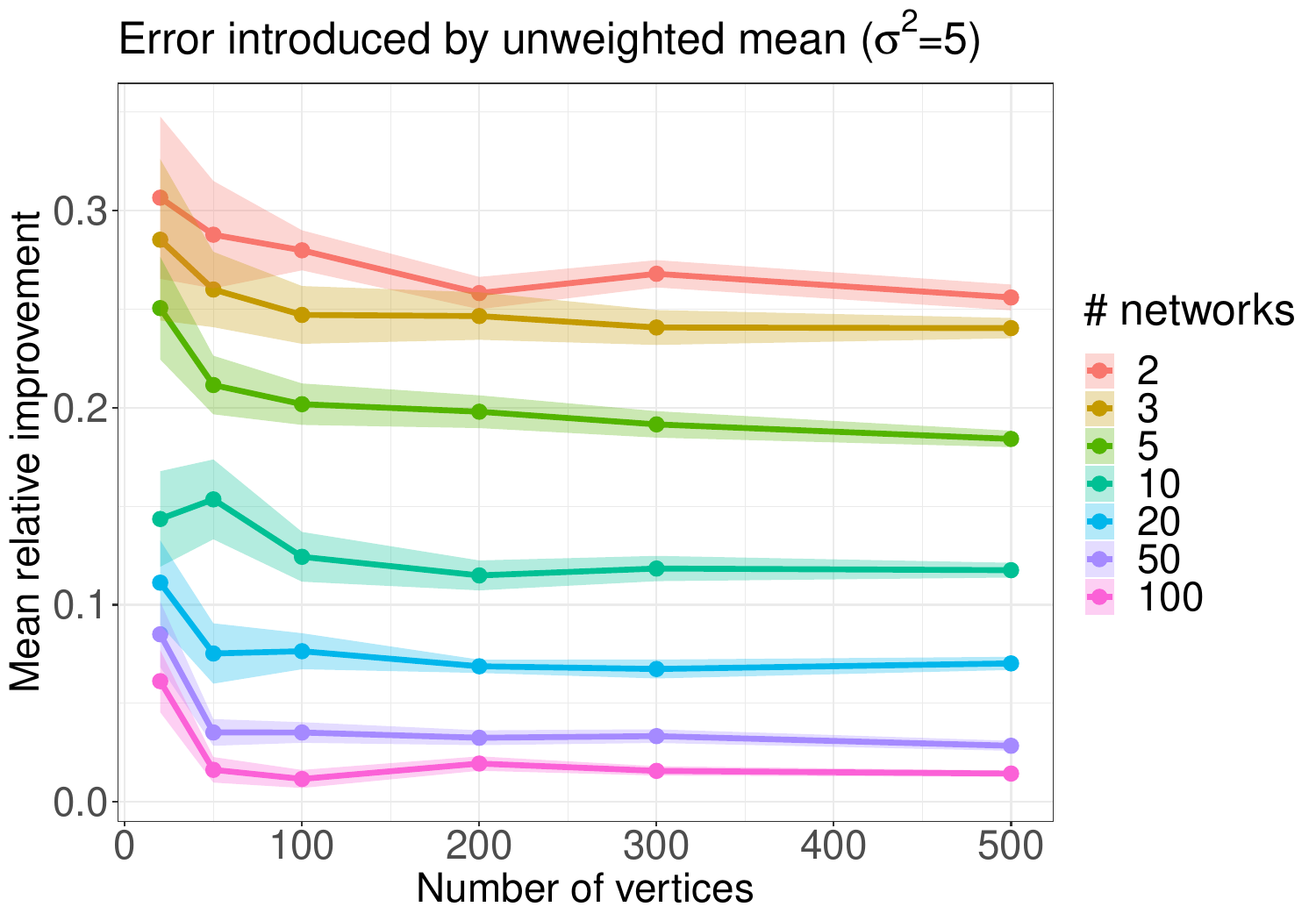} }
  \caption{\small Relative improvement of the eigenvalue truncation $\Phatwtd$
	of the weighted estimate $\Ahat$
	compared to its unweighted counterpart $\Phatunif$
	with Laplace-distributed edge noise.
	Each data point is the mean of 50 independent trials.
	(a) Relative improvement in Frobenius norm 
	as a function of the variance $\sigma_1^2$ of the outlier
	network for several values of the number of networks $N$,
	with the number of vertices $n=200$ fixed.
	(b) Relative improvement as a function of the number of vertices $n$ for
	different numbers of networks $N$, with outlier variance
	$\sigma_1^2 = 5$ fixed. }  
  \label{fig:laplacian}
\end{figure*}

\subsection{Application to neuroimaging data}

We briefly investigate how the choice of network
average impacts downstream analyses of real data.
We use the COBRE data set \citep{AineETAL2017},
a collection of fMRI scans from 69 healthy patients and
54 schizophrenic patients, for a total of $N=123$ subjects.
Each fMRI scan is processed to obtain a weighted graph
on $n=264$ vertices, in which each vertex represents a brain region, and
edge weights capture functional connectivity, as measured by
regional averages of voxel-level time series correlations.
The data are processed so that the brain regions align across subjects,
with the vertices corresponding to regions in the Power parcellation
\citep{PowerETAL2011}.

In real data,
we do not have access to the true low-rank matrix $P$,
if such a matrix exists at all.
Thus, to compare weighted and unweighted network averaging
on real-world data, we compare their impact on downstream tasks
such as clustering and hypothesis testing.  Even for these tasks, the ground truth is typically not known, and thus it is not possible to
directly assess which method returns a better answer.
Instead, we will check whether the weighted averages yield appreciably
different downstream results, and point to the synthetic experiments
as evidence that the weighted network average is likely the better choice.

We begin by examining
the effect of weighted averaging on estimated community structure.
We make the assumption once again that these networks share
a low-rank expectation $\E \AA{s} = P = X X^T$, with $X \in \R^{n \times d}$.
We will compare the behavior of clustering applied to the unweighted
network mean $\Phatunif$ against the behavior of clustering applied to its
weighted counterpart $\Phatwtd$.
In practice, the model rank $d$ is unknown and must be estimated from the data.
While this model selection task is important,
it is not the focus of the present work,
and thus instead of potentially introducing additional noise from imperfect estimation, we simply compare performance of the two estimators over a range of values of $d$.   For each fixed model rank $d$,
we first construct the estimate $\Xhat^{(d)} = \ASE( \Ahat, d )$,
and then estimate communities by applying $K$-means clustering
to the $n$ rows of $\Xhat^{(d)}$. 
Denote the resulting assignment of vertices to $d$ communities
by $\chat \in [K]^n$, and let $\cbar \in [K]^n$ denote the clustering
obtained by $K$-means applied to the rows of $\Xbar^{(d)} = \ASE( \Abar, d)$.    We measure the difference between these two assignments by  the discrepancy 
\begin{equation} \label{eq:discrep}
\delta(c,c')
= n^{-1} \min_{ \pi \in S_K }
        \sum_{i=1}^n \indicator\{ c_i \neq \pi( c'_i ) \},
\end{equation}
where $S_K$ denotes the set of all permutations of the set $[K]$.
This discrepancy measures the fraction of vertices that are assigned to
different communities by $c$ and $c'$ after accounting for
possible community relabeling.
The optimization over the set of permutations in~\eqref{eq:discrep}
can be solved using the Hungarian algorithm~\citep{Kuhn1955}.

Figure~\ref{fig:cobre:comms} shows the discrepancy
$\delta(\chat, \cbar)$ as a function of the number of communities $K$.
For simplicity, we take the number of communities equal to the model rank $d$,
though we note that similar patterns appear when we allow $K$ and $d$ to vary
separately.
In order to account for the possibility that the healthy and schizophrenic
populations display different community structures,
the three subplots of Figure~\ref{fig:cobre:comms}
show the results of the community estimation experiment just described
as applied only to the 69 healthy patients in the data set
(subplot a),
as applied only to the 54 schizophrenic patients
(subplot b)
and when pooling the healthy and schizophrenic patients
(subplot c).
Note that a similar pattern holds in all three of these cases.
Each data point in the figure is the mean of 20 independent runs
of $K$-means with random starting conditions,
with shaded regions indicating two standard errors.
It is clear from the plot that for a wide array of model choices,
the weighted and unweighted average networks
result in assigning a non-trivial fraction of the
vertices to different clusters.
Thus switching from unweighted to weighted averaging is likely
to have considerable effects on downstream inference tasks
pertaining to community structure.

\begin{figure*}
  \centering
  \subfloat[\label{subfig:cobre:comms:healthy}]{ \includegraphics[width=0.3\textwidth]{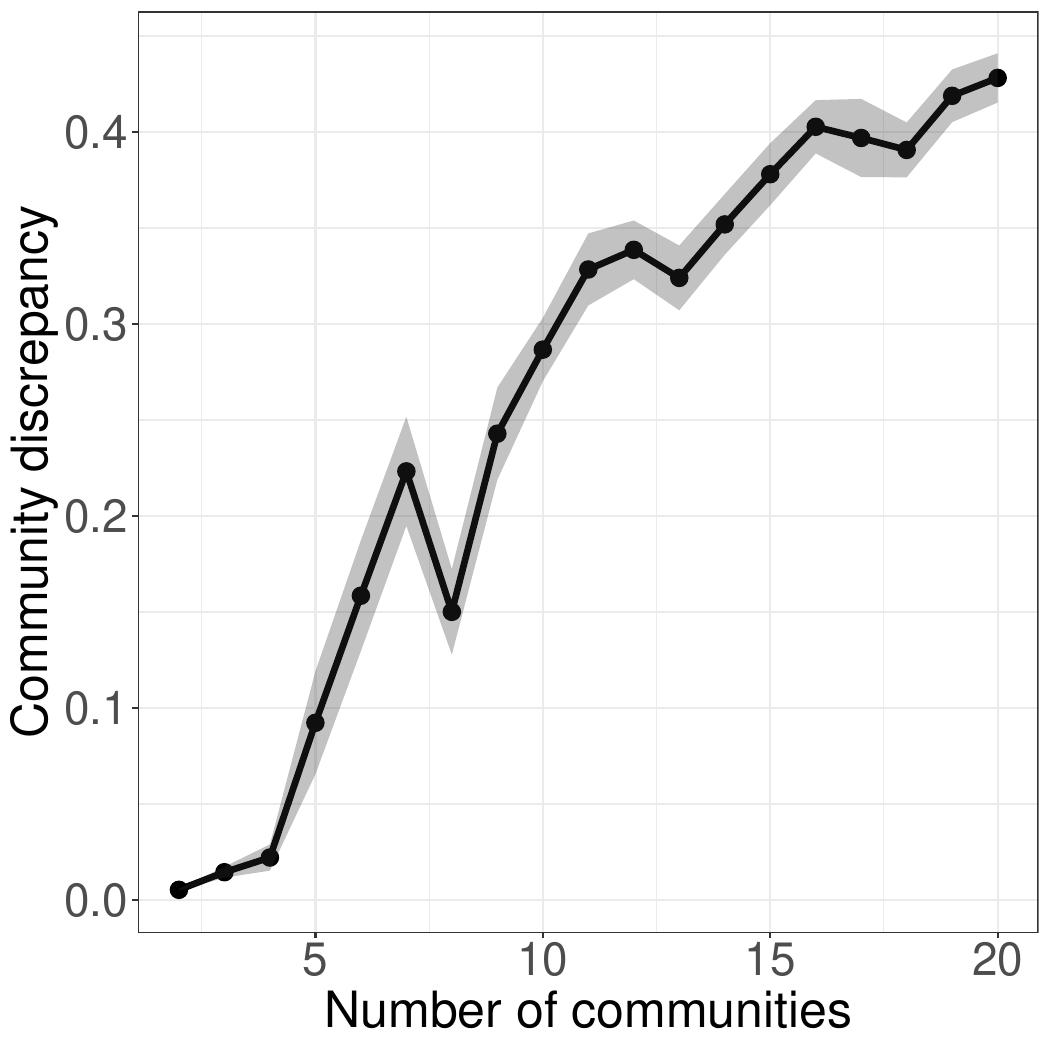} }
  \subfloat[\label{subfig:cobre:comms:schiz}]{ \includegraphics[width=0.3\textwidth]{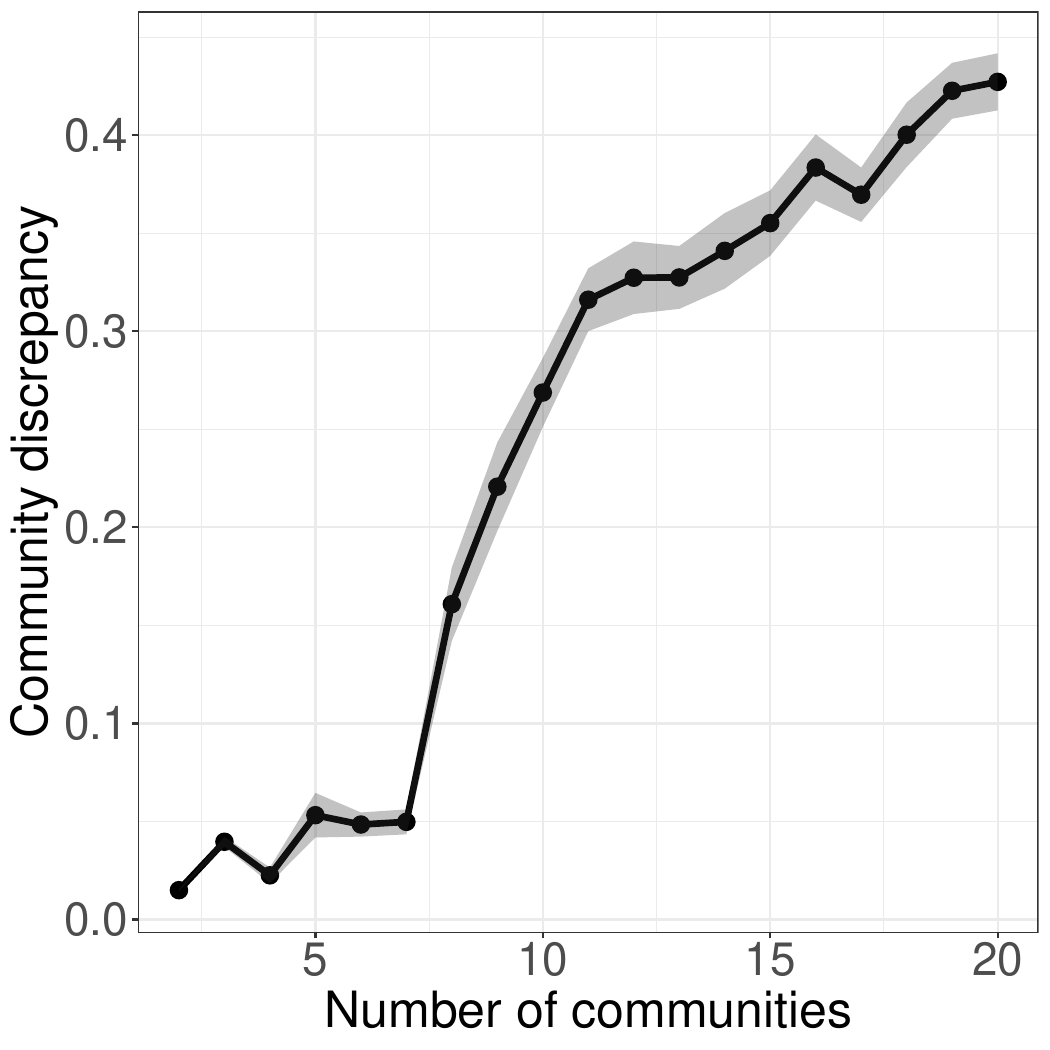} }
  \subfloat[\label{subfig:cobre:comms:full}]{ \includegraphics[width=0.3\textwidth]{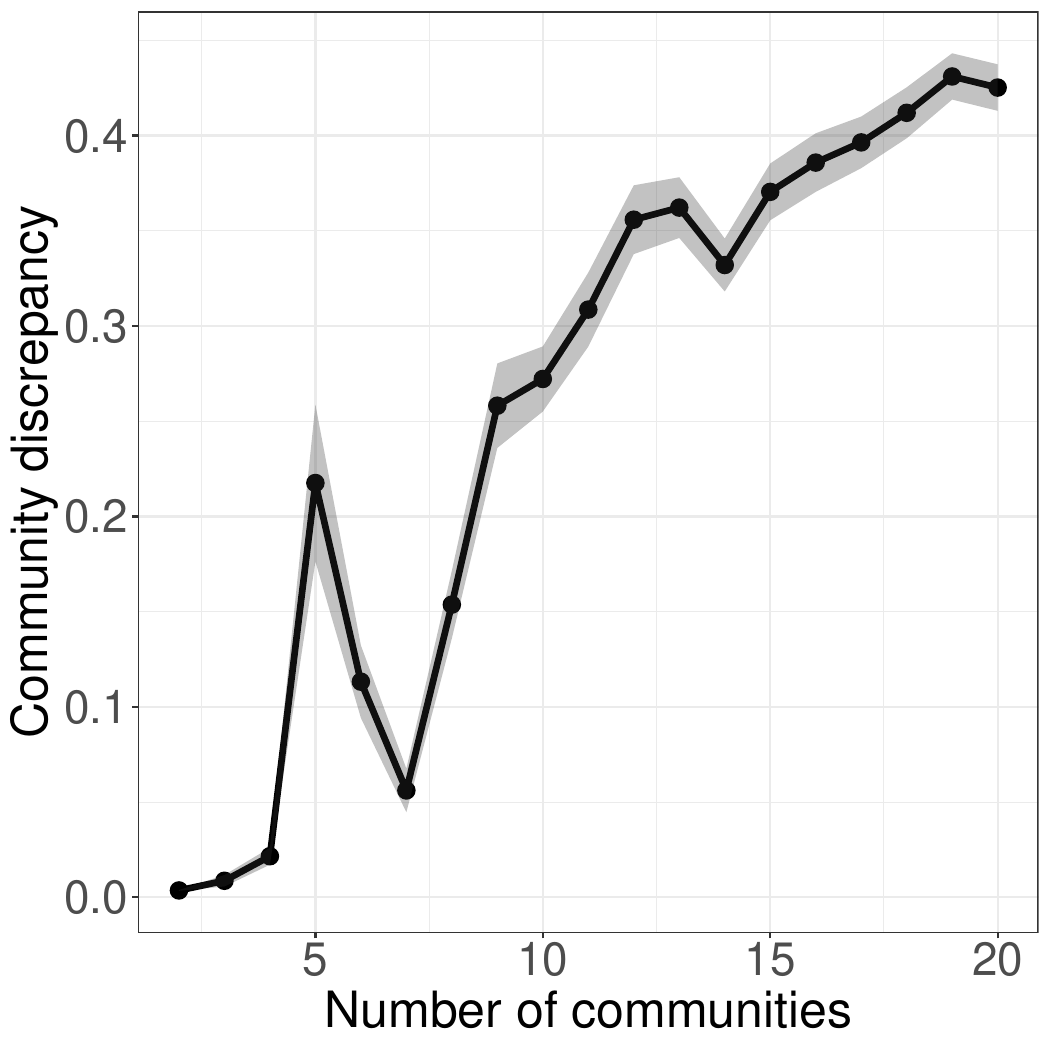} }

  \caption{ \small Fraction of vertices assigned to different communities when
        clustering the weighted and unweighted average networks
	based on the (a) healthy
	(b) schizophrenic
	and (c) pooled healthy and schizophrenic patients,
        as a function of the number of communities used.
	Each data point the mean of 20 independent
        trials, with shaded regions indicating two standard errors of the
        mean (randomness is present in this experiment due to starting
        conditions of the clustering algorithm). We see that over a broad
        range of model parameters (i.e., number of communities),
        the choice to use a weighted or unweighted network average results
        in different cluster assignments for between
	ten and forty percent of the vertices,
	and this pattern persists whether we pool all 123 patients
	or restrict our analysis to the healthy or schizophrenic patients.}
  \label{fig:cobre:comms}
\end{figure*}

The difference in task performance between these two different network averages
persists even for more complicated downstream inference tasks, as we now demonstrate.
The Power parcellation \citep{PowerETAL2011} is one of many ways of
assigning ROIs (i.e., nodes) to larger functional units of the brain,
typically called {\em functional regions}.
The Power parcellation assigns each of the 264 ROIs
(i.e., nodes) in the COBRE data set to one of 14 different communities,
corresponding to functional regions,
with sizes varying between approximately 5 and 50 nodes per community.
Table~\ref{table:power} summarizes the 14 functional regions
and their purported functions.
We refer to a pair of functional regions $(k,\ell)$, for every $k \le \ell$ as a network cell. Thus, the $K=14$ communities in the Power parcellation yield 105 cells.
For a given parcellation, a problem of scientific interest is to identify which network cells, if any, are different in schizophrenic patients compared to the healthy controls.
Such cells likely correspond to locations of functional differences between schizophrenic and healthy brains.
For concreteness, consider testing the hypothesis,
for each of the 105 possible cells $\{k,\ell\}$ for $1 \le k \le \ell \le 14$,
that the average functional connectivity within the cell is the same for the schizophrenic patients and the healthy controls.  These hypotheses can be tested using either weighted or unweighted network averages 
over the healthy and the schizophrenic samples, which we denote $\Phat^{(H)}$ and $\Phat^{(S)}$, respectively.
That is, letting $C_k$ denote the vertices associated
with the $k$-th functional region, we perform a two-sample $t$-test comparing the healthy sample cell mean 
$\{ \Phat^{(H)}_{i,j} : i \in C_k, j \in C_\ell \}$
to the schizophrenic sample cell
$\{ \Phat^{(S)}_{i,j} : i \in C_k, j \in C_\ell \}$,
for each pair $k \le \ell$.
Comparisons of this sort, with appropriate multiple testing correction,
are common in the neuroimaging literature.  
Nonetheless, we are not concerned here with whether or not precisely this
testing procedure is the most appropriate or most accurate method for
assessing differences between the schizophrenic and healthy populations.
Rather, we choose this procedure as representative of the
kinds of methods typically used for comparing network populations
in the literature, and our aim is to assess whether the use of weighted
or unweighted averaging leads to operationally different conclusions based
on the same data.

\begin{table}[h]
\centering
\footnotesize
\begin{tabular}{ r l r | r l r }
Region & Function & Nodes & Region & Function & Nodes \\
\hline
1  &  Uncertain & 28 & 8  &  Visual & 31 \\
2  &  Sensory/somatomotor Hand & 30 & 9  &  Fronto-parietal Task Control & 25 \\
3  &  Sensory/somatomotor Mouth & 5 & 10 &  Salience & 18 \\
4  &  Cingulo-opercular Task Control & 14 & 11 &  Subcortical & 13 \\
5  &  Auditory & 13  & 12 &  Ventral attention & 9 \\
6  &  Default mode & 58  & 13 &  Dorsal attention & 11 \\
7  &  Memory retrieval & 5 & 14 &  Cerebellar & 4 \\
\end{tabular}
\caption{\small Brain regions in the Power parcellation \citep{PowerETAL2011},
	their functions,
	and the number of nodes within each functional region.
	The region numbers correspond to those used in
	Figure~\ref{fig:cobre:2tests} below. }
\label{table:power}
\end{table}

The two subplots in Figure~\ref{fig:cobre:2tests} show the outcome of such
a comparison.
Each tile is colored according to the p-value returned by
a two-sample t-test comparing the estimated connection weights of the
schizophrenic and healthy patients within the corresponding cell.
Tiles highlighted by colored boxes correspond to cells for which the t-test
rejected at the $\alpha=0.01$ level after correcting for multiple comparisons
via the Benjamini-Hochberg procedure.
The left-hand subplot in Figure~\ref{fig:cobre:2tests}
shows the p-values for the unweighted test,
while the right-hand subplot shows the same procedure using the weighted
estimate instead of the unweighted estimate.
We see that after the Benjamini-Hochberg procedure, the weighted average
results in more rejections than the unweighted average,
and rejects a strict superset of the cells rejected by the unweighted
average.
Encouragingly, the cells identified as significant by both procedures
are fairly well localized, in that a few of the regions account
for most of the rejected cells (e.g., region 2 alone is associated
with twelve of the twenty seven cells rejected by both methods).
This suggests that the differences between the healthy and schizophrenic
populations are in fact localized to specific brain regions.
We note that the cells rejected by the weighted procedure and accepted
by the unweighted procedure are also in keeping with this localization,
in that all of the seven additional cells selected by the weighted procedure
are incident on at least one of regions 8, 9 or 13
(visual, fronto-parietal task control and dorsal attention,
respectively).

\begin{figure}[t!]
  \centering
  \subfloat[\label{subfig:cobre:2tests:unwtd}]{ \includegraphics[width=0.45\textwidth]{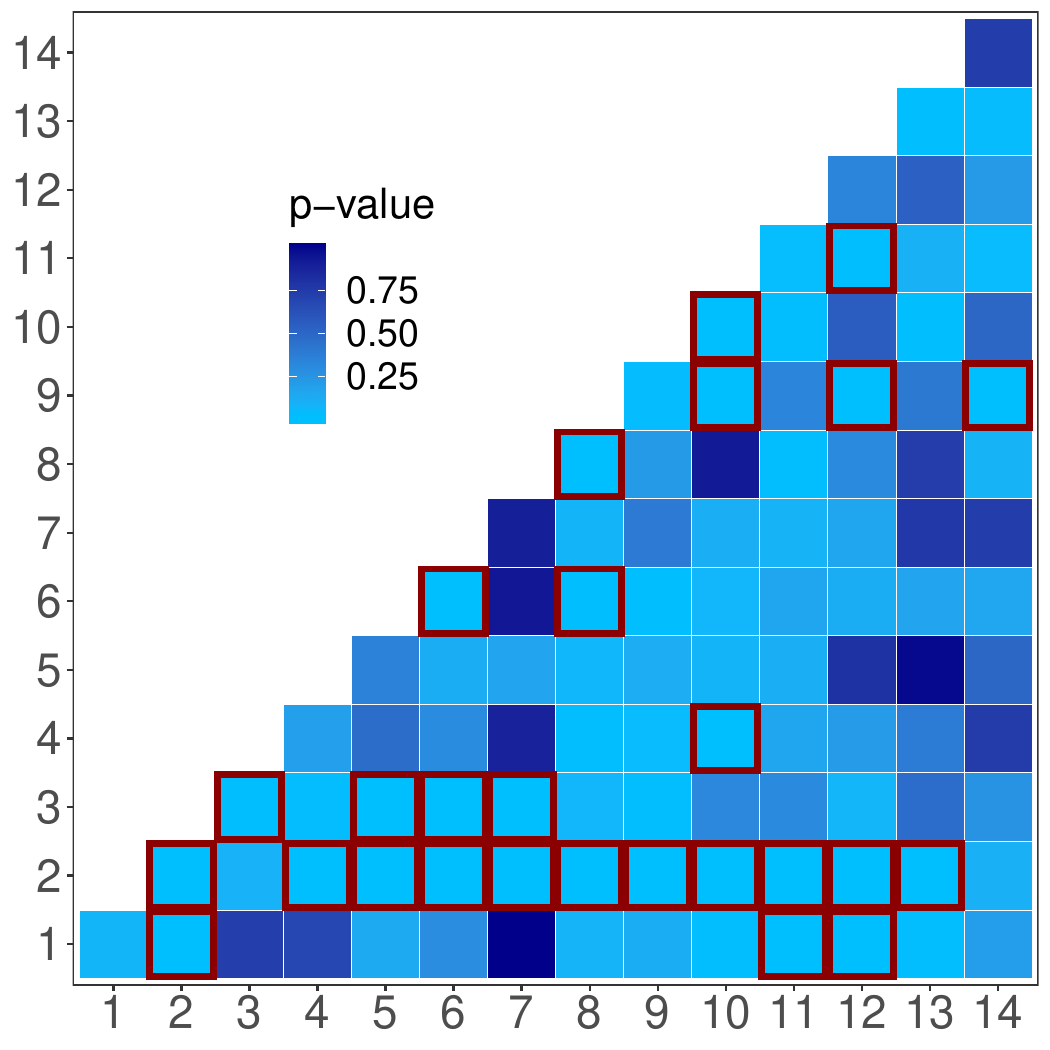} }
  \subfloat[\label{subfig:cobre:2tests:wtd}]{ \includegraphics[width=0.45\textwidth]{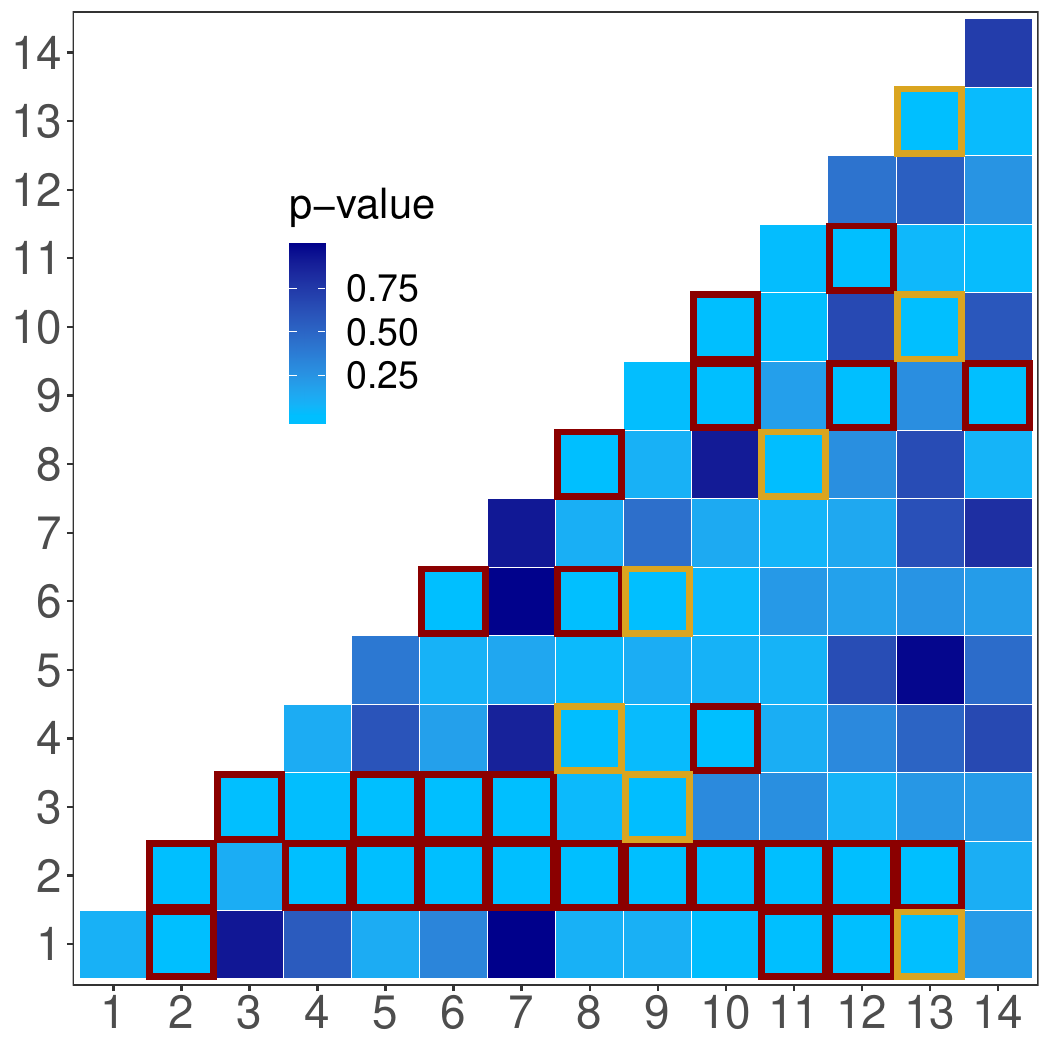} }
  \caption{\small Results of the cell-level significance tests of the unweighted (left) and weighted (right) network averages. Each tile corresponds to a cell, with tiles colored according to p-values. The cells highlighted by red or yellow squares indicate those for which one or both of the weighted and the unweighted procedures rejected the null hypothesis. Cells highlighted in red are those that were rejected by both the unweighted and weighted tests after Benjamini-Hochberg correction at false discovery rate 0.01. Cells highlighted in yellow in the weighted plot (right) correspond to those rejected by the weighted procedure, but not its unweighted counterpart. All cells rejected by the unweighted procedure were also rejected by the weighted procedure. Numbering of the axes corresponds to the brain region numbering given in Table~\ref{table:power} }
  \label{fig:cobre:2tests}
\end{figure}

The brain regions identified by these two procedures are consistent with existing work on the structural correlates of schizophrenia.
Most strikingly, we see exceptionally strong evidence that region 2
differs between schizophrenic and healthy patients.
This is in keeping with existing neuroscientific findings suggesting that
somatomotor processing is altered in schizophrenic patients
\citep[see, e.g.,][]{ShinnAberrant2015,KaufmannETAL2015,LiDysco2019,HummerFunctional2020}.
As another example, several cells involving the default mode network (region 6)
is involved in several of the cells identified by both methods.
This region, which is believed to be involved in undirected thought
(i.e., mind wandering), has been previously associated with
schizophrenia
\citep{BluETAL2007,WhitfieldGabrieliETAL2009,FoxETAL2015}.
It is interesting to note that this localized structure has emerged without any
explicit encoding of such a structure among the cells in the testing
procedure itself.

We note that the cell-level tests conducted by our two procedures outlined
above are likely to be dependent, owing to the fact
that each cell-level test incorporates edge-level information
that is likely to be correlated within each network.
The Benjamini-Yekutieli procedure, designed to control false discovery rate
under such dependence, yields qualitatively similar results to those seen
in Figure~\ref{fig:cobre:2tests}, though in that case, the unweighted
procedure rejects a superset of the cells rejected by the weighted procedure.
Broadly speaking, then, the choice to use a weighted or unweighted network
average has nontrivial consequences for downstream inference,
in that the procedures identify different sets of cells as being
implicated in the schizophrenia.
On the other hand, the weighted and unweighted procedures largely
agree in the cells that they identify as differing across the
two populations.
We conjecture that weighted network averaging will generally yield more
conservative results, leading to a smaller Type I error,
though our experiment just outlined shows that this need not be true uniformly.
We expect that these differences will mostly affect cells that are not clear-cut in either direction, and thus the specific problem and dataset at hand will determine how much the results differ.   At the same time, the non-clear cut cases are the most likely zone for new discoveries, and thus it is important to understand how different averaging choices affect the downstream analyses.  
Comparing the results from different averages can also be used as a measure of stability, increasing our confidence in conclusions when they agree \citep{Yu2013}.

\section{Summary and Discussion}
\label{sec:discussion}
We have presented an approach to handling heterogeneity in edge-level noise
for estimating shared structure from a collection of networks.
Under the setting where edge weights are i.i.d.\ Gaussian within the same
network, we have
shown that a weighted network average with weights proportional to estimated
variances of the edges is asymptotically equivalent to the maximum-likelihood
estimate in the case where the edge variances are known.
We have also presented a class of estimators
under weaker conditions on the tails, sub-Gaussian or sub-gamma instead of the Gaussian.
While showing theoretically
that these weighted estimates strictly improve upon unweighted network averages
is not easily done under these weaker assumptions,
synthetic experiments bear out the intuition that a weighted network average
based on estimated edge variances and/or scale parameters
improves upon a na\"ive unweighted sample mean of networks.
Further, experiments on real neuroimaging data showed that
the choice between weighted and unweighted network averaging has
consequences for downstream inference that cannot be ignored.

A most immediate avenue for future work is to pursue a more thorough
analysis of conditions  under which
weighted network averaging improves appreciably
upon unweighted averaging.   We are in the process of applying tools from random matrix theory to the multiple networks setting presented here.
Further afield, considering heavy-tailed distributions of network edges
is also of interest.   We also believe the techniques presented in the
present paper might be adapted to develop robust estimators
in network settings analogous to Huber's $\epsilon$-contamination model
\citep{Huber1964}.

{\bf Acknowledgements.} The authors acknowledge the support of the National Science Foundation, with KL and AL supported by DMS-1646108, and EL by  DMS-1916222. Additional support for KL was provided by the University of Wisconsin-Madison, Office of the Vice Chancellor for Research and Graduate Education with funding from the Wisconsin Alumni Research Foundation.


\newpage

\appendix

\section{sub-Gaussian and sub-gamma random variables}
\label{apx:taildefs}
For the sake of completeness, we state definitions and a few relevant facts on 
sub-Gaussian and sub-gamma random variables; these definitions are from \cite{BLM}, which can be consulted for a more thorough treatment.

\begin{definition}
  {\em
Let $Z$ be a random variable with $\E Z = 0$ and let
$\psi_Z(t) = \log \E e^{tZ}$ denote its cumulant generating function.
We say that $Z$ is {\em sub-Gaussian} with variance parameter $\nu \ge 0$
if for all $t \in \R$, $\psi_Z(t) \le t^2 \nu/2$.
}
\end{definition}

\begin{definition}
  {\em
Let $Z$ be a random variable with $\E Z = 0$ and let
$\psi_Z(t) = \log \E e^{tZ}$ denote its cumulant generating function.
Let $\nu,b \ge 0$.
We say that a random variable $Z$ is {\em sub-gamma} on the right tail
with parameter $(\nu,b)$ if $\psi_Z(t) \le \frac{t^2 \nu }{2(1-bt)}$
for all $t < 1/b$.
Similarly, we say that $Z$ is sub-gamma on the left tail
with parameter $(\nu,b)$ if $\psi_{-Z}(t) \le \frac{t^2 \nu}{2(1-bt)}$
for all $t < 1/b$.
If $Z$ is sub-gamma on both the left and the right tails with parameter
$(\nu,b)$, then we say that $Z$ is sub-gamma with parameter $(\nu,b)$,
and write that $Z$ is $(\nu,b)$-sub-gamma.
} \end{definition}

A basic property of sub-Gaussian random variables is that their sum is sub-Gaussian:  if $\{ Z_i \}_{i=1}^m$ are independent sub-Gaussian variables
with variance parameters $\nu_i$,
and $\{ \alpha_i \}_{i=1}^m$ are real numbers,
then $\sum_{i=1}^m \alpha_i Z_i$ is sub-Gaussian with variance parameter
$\sum_{i=1}^m \alpha_i^2 \nu_i$.

Similarly, if  $\{ Z_i \}_{i=1}^m$ are independent sub-gamma variables
 with parameters $(\nu_i,b_i)$ for all $i \in [m]$, $\sum_{i=1}^m \alpha_i Z_i$ is sub-gamma with parameter
$( \sum_{i=1}^m \alpha_i^2 \nu_i, \max_i |\alpha_i| b_i)$.

Some references use the term {\em sub-exponential} for the
tail behavior just defined. We instead reserve this term for the special case obtained by taking $\nu=\lambda^2$, $b=0$.
\begin{definition} {\em
A random variable $Z$ with $\E Z = 0$ is called  sub-exponential with parameter
$\lambda > 0$ if its MGF satisfies
$\E \exp\{ t Z \} \le \exp\{ t^2 \lambda^2 / 2 \}$
whenever $|t| \le 1/\lambda$.
} \end{definition}

If a random variable $Z$ is sub-Gaussian with parameter $\nu$, then
$Z^2-\E Z^2$ is sub-exponential with parameter $16\nu$; 
see, for example, Lemma 1.12 in the lecture notes by \cite{RigHut2017}.

\section{Proofs and Technical Results}
In what follows, we prove our main results.
\subsection{Proof of Lemma~\ref{lem:genericbound} (Concentration in $\tti$-norm)}
\label{apx:genericbound}
Lemma~\ref{lem:genericbound} 
generalizes and extends results of 
\cite{LyzTanAthParPri2017} and \cite{LevAthTanLyzPri2017}.
To begin with, we require two technical results.
The first is a modification of 
Proposition 16 in \cite{LyzTanAthParPri2017}
to be agnostic to the growth of the spectrum of $P$
instead of assuming a lower-bound on the growth rate of
the non-zero eigenvalues of $P$.
The proof is otherwise identical and is omitted.
We note that this result is entirely deterministic,
but we will apply it below when $M$ is a random matrix with expectation $P$.
\begin{proposition} \label{prop:innerprods}
Let $M,P \in \R^{n \times n}$
with $P = X X^T$ for some $X \in \R^{n \times d}$.
Let $P = \UP \SP \UP^T$ be the rank-$d$ singular value decomposition of $P$
and define $\UM \in \R^{n \times d}, \SM \in \R^{d \times d}$
so that $\UM \SM \UM^T$ is the rank-$d$ eigenvalue truncation of $M$.
Let $V_1 D V_2^T$ be the rank-$d$ singular value decomposition of $\UP^T \UM$.
Then
\begin{equation*}
  \| \UP^T \UM - V_1 V_2^T \|_F \le \frac{ d \| M - P \|^2 }{ \lambda_d^2(P) }.
\end{equation*}
\end{proposition}

Our second technical result is an adaptation of
Lemma 17 in \cite{LyzTanAthParPri2017}
and Lemma 4 in \cite{LevAthTanLyzPri2017},
again adapted to the case where no growth assumptions on $\lambda_d(P)$
are made.
The proof is a straight-foward application of Davis-Kahan style bounds
\citep[e.g., Theorem 2 in][]{YuWanSam2015}
and a modification of the argument in \cite{LevAthTanLyzPri2017},
and details are thus omitted.

\begin{proposition} \label{prop:Vapproxcomm}
Let $P \in \R^{n \times n}$
with $P = X X^T$ for some $X \in \R^{n \times d}$
and $P = \UP \SP \UP^T$ the rank-$d$ singular value decomposition of $P$.
Let $M \in \R^{n \times n}$ be random
with rank-$d$ eigenvalue truncation $\UM \SM \UM^T$,
where $\UM \in \R^{n \times d}, \SM \in \R^{d \times d}$.
Then
\begin{equation} \label{eq:projbound}
\| \UM - \UP \UP^T \UM \|_F \le \frac{C\sqrt{d} \| M - P \|}{\lambda_d(P)}.
\end{equation}
Further,
suppose that there exists a constant $c_0 \in [0,1)$ such that
with probability at least $p_0$, $\| M - P \| \le c_0 \lambda_d(P)$.
Let $V_1 D V_2^T$ be the rank-$d$ singular value decomposition of $\UP^T \UM$,
and define $V = V_1 V_2^T$.
Then with probability at least $p_0$,
\begin{equation} \label{eq:Vapproxcomm:1}
\| V \SM - \SP V \|_F
\le \frac{ C d \|M-P\|^2 \kappa(P) }{ \lambda_d(P) }
  + \| \UP^T(M-P)\UP \|_F
\end{equation}
\begin{equation} \label{eq:Vapproxcomm:2}
\| V \SM^{1/2} - \SP^{1/2} V \|_F
\le \frac{ C \| V \SM - \SP V \|_F }{ \lambda_d^{1/2}(P) } ~~~\text{ and }
\end{equation}
\begin{equation} \label{eq:Vapproxcomm:3}
\| V \SM^{-1/2} - \SP^{-1/2} V \|_F \le
  \frac{ C \| V \SM - \SP V \|_F }{ \lambda_d^{3/2}(P) }.
\end{equation}

\end{proposition}

We are now equipped to prove Lemma~\ref{lem:genericbound}.

\begin{delayedproof}{Proof of Lemma~\ref{lem:genericbound}}
Let $V \in \R^{d \times d}$ be the orthogonal matrix defined above
in Proposition~\ref{prop:Vapproxcomm}
and define the following three matrices:
\begin{equation*} \begin{aligned}
R_1 &= \UP \UP^T \UM - \UP V , \\
R_2 &= V \SM^{1/2} - \SP^{1/2} V  , \\
R_3 &= \UM - \UP \UP^T \UM + R_1 = \UM - \UP V.
\end{aligned} \end{equation*}
Adding and subtracting appropriate quantities,
\begin{equation*} \begin{aligned}
\UM \SM^{1/2} - \UP \SP^{1/2} V &= (M -P)\UP\SP^{-1/2}V
	+ (M-P)\UP(V\SM^{-1/2} - \SP^{-1/2} V) \\
  	&~~~+ \UP \UP^T (M - P)\UP V\SM^{-1/2}
		+ R_1 \SM^{1/2} + \UP R_2 \\
	&~~~+ (I-\UP\UP^T)(M-P)R_3 \SM^{-1/2}.
\end{aligned} \end{equation*}
Applying the triangle inequality and the fact that the Frobenius norm
is an upper bound on the $(\tti)$-norm, we have 
\begin{equation} \label{eq:bigsum}
\begin{aligned}
\| \UM &  \SM^{1/2} - \UP \SP^{1/2} V \|_{\tti}
\le \| (M -P)\UP\SP^{-1/2}V \|_{\tti} \\
  &+ \| (M-P)\UP(V\SM^{-1/2} - \SP^{-1/2} V) \|_F
  		+ \| \UP \UP^T (M - P)\UP V\SM^{-1/2} \|_F \\
  &+ \| (I-\UP\UP^T)(M-P)R_3 \SM^{-1/2} \|_F
	+ \| R_1 \SM^{1/2} \|_F + \| \UP R_2 \|_F.
\end{aligned}  \end{equation}
We will bound each of these summands in turn.
Firstly, by definition of the spectral and $(\tti)$ norms,
\begin{equation} \label{eq:bigsum:term1}
  \| (M -P)\UP\SP^{-1/2}V \|_{\tti}
  \le \| (M -P)\UP \|_{\tti} \| \SP^{-1/2} \|
  \le \frac{ \| (M -P)\UP \|_{\tti} }{ \lambda_d^{1/2}(P) }.
\end{equation}
The second term on the right-hand side of~\eqref{eq:bigsum}
is bounded as
\begin{equation*} \begin{aligned}
\| (M-P)\UP(V\SM^{-1/2} - \SP^{-1/2} V) \|_F
&\le \| M - P \| \| \UP \| \| V\SM^{-1/2} - \SP^{-1/2} V \|_F \\
&= \| M - P \| \| V\SM^{-1/2} - \SP^{-1/2} V \|_F,
\end{aligned} \end{equation*}
from which Proposition~\ref{prop:Vapproxcomm} implies
that with probability at least $p_0$,
\begin{equation} \label{eq:bigsum:term2}
\begin{aligned}
 \| (M-P)&\UP(V\SM^{-1/2} - \SP^{-1/2} V) \|_F 
\le \frac{ C \| M - P \| \| V \SM - \SP V \|_F }
        { \lambda_d^{3/2}(P) }.
\end{aligned} \end{equation}

By the assumption in Equation~\eqref{eq:assum:specbound},
with probability at least $p_0$,
\begin{equation} \label{eq:SMbound}
\| \SM^{-1/2} \| \le \left( \lambda_d(P) - \|M-P\| \right)^{-1/2}
\le C\lambda_d^{-1/2}(P).
\end{equation}
Thus, the third term on the right-hand side of~\eqref{eq:bigsum} satisfies
\begin{equation} \label{eq:bigsum:term3}
\begin{aligned}
\| \UP \UP^T (M &- P)\UP V\SM^{-1/2} \|_F
  \le \| \UP \| \| \UP^T (M - P)\UP \|_F \| V \| \| \SM^{-1/2} \| \\
 &\le \| \UP^T (M - P)\UP \|_F \| \SM^{-1/2} \| \\
& \le \frac{ \| \UP^T(M - P)\UP \|_F }
        { \sqrt{ \lambda_d(P) - \|M-P\| } }
\le \frac{ C\| \UP^T(M - P)\UP \|_F }{ \lambda_d^{1/2}(P) },
\end{aligned} \end{equation}
where the penultimate inequality follows from Equation~\eqref{eq:SMbound}
and the last inequality holds with probability at least $p_0$ by
the assumption in Equation~\eqref{eq:assum:specbound}. 

Considering the fourth term on the
right-hand side of~\eqref{eq:bigsum}, recall that $R_3 = \UM - \UP V$,
so that by adding and subtracting appropriate quantities we have
\begin{equation} \label{eq:bigsum:term4:expand}
\begin{aligned}
(I-\UP\UP^T)&(M-P)R_3 \SM^{-1/2} \\
&= (I-\UP\UP^T)(M-P)(\UM - \UP \UP^T \UM)\SM^{-1/2} \\
 &~~~~~~+ (I-\UP\UP^T)(M-P)((I-\UP\UP^T)(M-P) - \UP V)\SM^{-1/2}.
\end{aligned} \end{equation}
The former of these quantities is bounded as
\begin{equation*} \begin{aligned}
\| (I&-\UP\UP^T)(M-P)(\UM - \UP \UP^T \UM)\SM^{-1/2} \|_F \\
&\le \| I-\UP\UP^T \| \| M - P \| \| \UM - \UP\UP^T\UM \|_F \| \SM^{-1/2} \|
\le \frac{ C \sqrt{d} \| M - P \|^2 }{ \lambda_d^{3/2}(P) }
\end{aligned} \end{equation*}
where we have used Equations~\eqref{eq:projbound} and~\eqref{eq:SMbound}
to obtain the second inequality.
The second term on the right-hand side of~\eqref{eq:bigsum:term4:expand}
is bounded by Proposition~\ref{prop:innerprods}
and Equation~\eqref{eq:SMbound} as
\begin{equation*} \begin{aligned}
\| (I-\UP\UP^T)&(M-P)((I-\UP\UP^T)(M-P) - \UP V)\SM^{-1/2} \|_F \\
&\le \|I-\UP\UP^T \| \|M-P \| \| \UP \| \| \UP^T\UM - V \|_F \| \SM^{-1/2} \|
\le \frac{ C d \| M-P \|^3 }{ \lambda_d^{5/2}(P) },
\end{aligned} \end{equation*}
Thus, combining the above two displays with
Equation~\eqref{eq:bigsum:term4:expand},
\begin{equation} \label{eq:bigsum:term4}
\begin{aligned}
\| (I-\UP\UP^T)&(M-P)R_3 \SM^{-1/2} \|_F \\
&\le \frac{ C \sqrt{d} \| M - P \|^2 ( \lambda_d(P) + \sqrt{d}\| M-P \| ) }
        { \lambda_d^{5/2}(P) }
\le \frac{ C d \| M-P \|^2 }{ \lambda_d^{3/2}(P) },
\end{aligned} \end{equation}
where the last inequality holds with probability at least $p_0$
by the assumption in Equation~\eqref{eq:assum:specbound}.

Finally, we bound the last two summands on the right-hand side
of~\eqref{eq:bigsum}.
Recall that $R_1 = \UP \UP^T \UM - \UP V = \UP( \UP^T \UM - V)$,
Proposition~\ref{prop:innerprods} yields
$$ \| R_1 \|_F \le \| \UP \| \| \UP^T \UM - V \|_F
  \le \frac{ C d \| M - P \|^2 }{ \lambda_d^2( P ) }, $$
whence by the assumption in Equation~\eqref{eq:assum:specbound},
\begin{equation} \label{eq:bigsum:term5}
\| R_1 \SM^{1/2} \|_F \le \| R_1 \|_F \| \SM^{1/2} \|
\le \frac{ C d \| M - P \|^2 \lambda_1^{1/2}(P) }
	{ \lambda_d^2(P) }.
\end{equation}
Recalling $R_2 = V \SM^{1/2} - \SP^{1/2} V$,
Proposition~\ref{prop:Vapproxcomm} implies that
\begin{equation} 
\| R_2 \|_F \le \frac{ C \| V \SM - \SP V \|_F }{ \lambda_d^{1/2}(P) }.
\end{equation}
Applying this bound along with 
Equations~\eqref{eq:bigsum:term1},~\eqref{eq:bigsum:term2},
~\eqref{eq:bigsum:term3},~\eqref{eq:bigsum:term4} and~\eqref{eq:bigsum:term5}
to Equation~\eqref{eq:bigsum}, collecting terms
and making use of the assumption in Equation~\eqref{eq:assum:specbound} once more,
\begin{equation*} \begin{aligned}
\| \UM \SM^{1/2} - \UP \SP^{1/2} V \|_{\tti} 
&\le
\frac{ \| (M -P)\UP \|_{\tti} }{ \lambda_d^{1/2}(P) }
+
\frac{ C \| V \SM - \SP V \|_F }{ \lambda_d^{1/2}(P) }  \\
&~~~~~~+ \frac{ C \| \UP^T(M-P) \UP \|_F }{ \lambda_d^{1/2}(P) }
+ \frac{ Cd\|M-P\|^2 }{ \lambda_d^{3/2}(P) }
	\left( 1 + \kappa^{1/2}(P) \right).
\end{aligned} \end{equation*}
Applying Proposition~\ref{prop:Vapproxcomm} to
$\| V \SM - \SP V \|_F$,
\begin{equation*} \begin{aligned}
\| \UM \SM^{1/2} - \UP \SP^{1/2} V \|_{\tti} 
&\le \frac{ \| (M -P)\UP \|_{\tti} }{ \lambda_d^{1/2}(P) }
+ \frac{ C \| \UP^T(M-P)\UP \|_F }{ \lambda_d^{1/2}(P) } \\
&~~~~~~+ \frac{ Cd\|M-P\|^2 }{\lambda_d^{3/2}(P)}
  \left( 1 + \kappa^{1/2}(P) + \kappa(P) \right).
\end{aligned} \end{equation*}
Noting that $\kappa(P) \ge 1$ completes the proof.
\end{delayedproof}

\subsection{Proof of Theorem~\ref{thm:subexpXtilde} (Concentration under sub-gamma assumptions)}
\label{apx:subexpXtilde}
Using the results presented above, we are now able to prove Lemma~\ref{lem:mxspecbound} and Theorem~\ref{thm:subexpXtilde}.
Our proof of Theorem~\ref{thm:subexpXtilde} relies on applying Lemma~\ref{lem:genericbound} to the case of $M = \Atilde$.
This in turn requires a bound on the spectral norm error that is central to Lemma~\ref{lem:genericbound}.
This spectral norm error is controlled by Lemma~\ref{lem:mxspecbound},
the proof of which relies on the following  matrix Bernstein bound.

\begin{theorem}[\cite{Tropp2012} Theorem 6.2]
\label{thm:Tropp2012:Bern}
Let $\{ Z_k \}$ be a finite sequence of independent, random, self-adjoint
$n$-by-$n$ matrices each satisfying $\E Z_k = 0$ and
$ \E Z_k^p \preceq p! R^{p-2} M_k^2/2$ for $p=2,3,\dots$,
where $R \in \R$ and $\{ M_k \} \subset \R^{n \times n}$
are deterministic matrices
and $\preceq$ denotes the semidefinite ordering,
Define $\mxbernsymbol^2 = \| \sum_k M_k^2 \|$. Then for all $t \ge 0$,
\begin{equation*}
\Pr\left[ \Big\| \sum_k Z_k \Big\| \ge t \right]
        \le n \exp\left\{ \frac{ -t^2 }{  2\mxbernsymbol^2 + 2Rt } \right\}.
\end{equation*}
\end{theorem}

\begin{delayedproof}{Proof of Lemma~\ref{lem:mxspecbound}}
To apply Theorem~\ref{thm:Tropp2012:Bern}, we first decompose
$\sum_{s=1}^N w_s \AA{s}-P$
into a sum of independent zero-mean symmetric matrices.
Letting $\{e_1,e_2,\dots,e_n\}$ denote the standard basis vectors,
define, for all $1 \le i \le j \le n$ the $n$-by-$n$ matrices
\begin{equation} \label{eq:Bmx}
 \bB_{i,j} =
   \begin{cases} e_i e_j^T + e_j e_i^T & \mbox{ if } 1 \le i < j \le n, \\
                e_i e_i^T &\mbox{ if } i = j, \end{cases}
\end{equation}
and the $n$-by-$n$ random matrices
$\bZ_{s,i,j} = (w_s \AA{s}-P)_{i,j} \bB_{i,j}$ for all $1 \le i \le j \le n$
and all $s \in [N]$.
We bold these matrices in this proof to remind the reader
that indexing by $s,i,j$ is specifying one of $O(N n^2)$ matrices,
rather than a matrix entry.
The set $\{ \bZ_{s,i,j} : s \in [N], 1 \le i \le j \le n \}$ sum to
$\sum_s w_s \AA{s}-P$, and
satisfy the symmetry, independence and unbiasedness assumptions required of
Theorem~\ref{thm:Tropp2012:Bern}.
By the assumption that $(\AA{s}-P)_{i,j}$ is
$(\nu_{s,i,j},b_{s,i,j})$-sub-gamma,
we have that $w_s(\AA{s}-P)_{i,j}$ is
$(w_s^2 \nu_{s,i,j}, w_s b_{s,i,j})$-sub-gamma. 
A moment-bounding argument
similar to that in Theorem 2.3 of~\cite{BLM} yields that
\begin{equation*} \begin{aligned}
\E |w_s &(\AA{s} - P)_{i,j}|^p
= \int_0^\infty p u^{p-1} \Pr[ |w_s(\AA{s}-P)|_{i,j} > u ] du \\
&\le 2p\int_0^\infty \left(\sqrt{2w_s^2 \nu_{s,i,j} t} + w_s b_{s,i,j}t\right)^{p-1}
          \frac{e^{-t}(w_s\sqrt{2\nu_{s,i,j} t} + 2w_s b_{s,i,j} t)}{2t} dt \\
&\le 2^{p-1} p w_s^p \int_0^\infty \big[\big(2\nu_{s,i,j}\big)^{p/2} t^{p/2 - 1} + (2 b_{s,i,j})^p t^{p-1}\big] e^{-t} dt  \\
&= 2^{p-1} p w_s^p [ (2\nu_{s,i,j})^{p/2} \Gamma(p/2) + (2b_{s,i,j})^p \Gamma(p) ] \\
&\le 2^p p! \, w_s^p (\sqrt{2 \nu_{s,i,j}} + 2b_{s,i,j})^p.
\end{aligned} \end{equation*}

Thus, defining $\eta_{s,i,j} = 2(\sqrt{2\nu_{s,i,j}} + 2b_{s,i,j})$
and $R = \max_{s,i,j} \eta_{s,i,j}$, we have
\[
\E \bZ_{s,i,j}^p = \E \left[ w_s(\AA{s}-P)_{i,j} \right]^p \bB_{i,j}^p
        \preceq p! w_s^p \eta_{s,i,j}^p \bB_{i,j}^p
        \preceq \frac{p!}{2} R^{p-2} (\sqrt{2} w_s \eta_{s,i,j} \bB_{i,j})^2, 
\]
since $\bB_{i,j}^p = \bB_{i,j}^{2 - \lfloor p/ 2\rfloor} \preceq \bB_{i,j}^2$.
Taking the bounding matrices $M_k$ in Theorem~\ref{thm:Tropp2012:Bern}
to be
$\{ \sqrt{2} w_s \eta_{s,i,j} \bB_{i,j} : 1 \le s \le N,
					 1 \le i \le j \le n\}$,
$\mxbernsymbol$ in Theorem~\ref{thm:Tropp2012:Bern} becomes
\begin{equation*} \begin{aligned}
\mxbernsymbol^2 &= \left\| \sum_{s=1}^N w_s^2
		\sum_{1 \le i \le j \le n} 2\eta_{s,i,j}^2 \bB_{i,j}^2
		\right\| \\
        &= 2 \left\| \sum_{s=1}^N \sum_{i=1}^n w_s^2 \eta_{s,i,i}^2 e_ie_i^T
        + \sum_{s=1}^N \sum_{1 \le i < j \le n}
		w_s^2 \eta_{s,i,j}^2 (e_i e_i^T + e_j e_j^T) \right\| \\
  &= 2 \max_i \sum_{s=1}^N \sum_{j=1}^n w_s^2 \eta_{s,i,j}^2.
\end{aligned} \end{equation*}
Applying Theorem~\ref{thm:Tropp2012:Bern}, we have
\begin{equation*}
\Pr\left[ \bigg\| \sum_{s=1}^N \sum_{1 \le i \le j \le n} \bZ_{s,i,j} \bigg\|
		\ge t \right]
 \le n\exp\left\{ \frac{ -t^2/2 }{ \mxbernsymbol^2 + R t } \right\}.
\end{equation*}
Taking $t = 3(\sqrt{2\mxbernsymbol^2} + 3R)\log n$ is enough to ensure that
$ t^2 \ge 6(\mxbernsymbol^2 + Rt)\log n$, so that
\begin{equation*}
\Pr\left[ \left\| \sum_{s=1}^N \sum_{1 \le i \le j \le n}
		\bZ_{s,i,j} \right\|
		\ge
 		3(\sqrt{2\mxbernsymbol^2} + 3R)\log n
	\right]
        \le n^{-2}.
\end{equation*}
Note that we have
\begin{equation*}
R = \max_{i,j,s} \eta_{s,i,j} w_s 
= \sqrt{ \max_{i,j,s} (\eta_{s,i,j} w_s )^2 }
\le \sqrt{ \max_{i \in [n]}
	\sum_{j=1}^n \sum_{s=1}^N (\eta_{s,i,j} w_s)^2 }
= \sqrt{\mxbernsymbol^2 / 2},
\end{equation*}
so our upper bound can be further simplified to
\begin{equation*}
 \left\| \sum_{s=1}^N w_s \AA{s} - P \right\|
\le \frac{ 15 \sqrt{2\mxbernsymbol^2} }{ 2 }\log n,
\end{equation*}
which completes the proof.
\end{delayedproof}

Recall that our aim in proving Theorem~\ref{thm:subexpXtilde}
 is to establish a bound on the
$(2,\infty)$-norm of $\Xtilde - XV$,
for a suitably-chosen orthogonal matrix $V \in \R^{d \times d}$,
where $\Atilde = \sum_{s=1}^N w_s \AA{s}$ and $\Xtilde = \ASE( \Atilde, d)$.
By Lemma~\ref{lem:genericbound}, in order to bound this norm,
it suffices to control $\|\Atilde-P\|$, $\|\UP^T(\Atilde-P)\UP\|_F$,
$\| (\Atilde-P)\UP \|_{\tti}$ and the spectrum of $P$.
Lemma~\ref{lem:mxspecbound} controls the first of these terms,
while Propositions~\ref{prop:UTAPU} and~\ref{prop:APtti}
control the Frobenius and $(2,\infty)$-norms, respectively.

\begin{proposition} \label{prop:UTAPU}
Let $\{ \AA{s} \}_{s=1}^N$ be independent and for all $s=1,2,\dots,N$,
$\{ (\AA{s}-P)_{i,j} : 1 \le i \le j \le n \}$ are independent
and $(\AA{s}-P)_{i,j}$ is $(\nu_{s,i,j}, b_{s,i,j})$-sub-gamma.
Assume that $P = \E \AA{s}$ is rank $d$ with $d = O(n)$,
and let $U_{j,k}$ denote the $(j,k)$-element
of the matrix $\UP \in \R^{n \times d}$ of eigenvectors of $P$
with non-zero eigenvalues.
Define for all $k,\ell \in [d]$ the quantities
\begin{equation*}
\nubar_{k,\ell} = \sum_{s=1}^N \sum_{i=1}^n \sum_{j=1}^n
		w_s^2 U_{i,k}^2 U_{j,\ell}^2 \nu_{s,i,j},
~~~~~~
\bbar_{k,\ell} = \max_{s\in[N],i,j \in [n]} w_s |U_{i,k} U_{j,\ell}| b_{s,i,j}.
\end{equation*}
Then for any fixed collection of weights
$\{ w_s \}_{s=1}^N$ with $w_s \ge 0$ and $\sum_{s=1}^N w_s = 1$,
\begin{equation*}
\left\| \UP^T \left(\sum_{s=1}^N w_s \AA{s}-P \right)\UP \right\|_F \le
2\left( \sum_{k=1}^d \sum_{\ell=1}^d
	(\sqrt{2\nubar_{k,\ell}} + 2\bbar_{k,\ell})^2
  \right)^{1/2} \log n
\end{equation*}
Further, when the top eigenvectors of $P$ delocalize so that
$|U_{i,k}| \le Cn^{-1/2}$ for all $i \in [n], k \in [d]$,
it holds with probability at least $1-Cn^{-2}$ that
\begin{equation*}
\left\| \UP^T\left(\sum_{s=1}^N w_s \AA{s}-P \right)\UP \right\|_F \le
\frac{ Cd }{n} \left( \sum_{s=1}^N \sum_{i=1}^n \sum_{j=1}^n w_s^2 \nu_{s,i,j}
	+ 2\max_{s \in [N],i,j \in[n]} w_s^2 b_{s,i,j}^2 \right)^{1/2}
	\log n.
\end{equation*}
\end{proposition}
\begin{proof}
Fix $k,\ell \in [d]$ and consider
$[ \UP^T(\sum_{s=1}^N w_s \AA{s}-P)\UP ]_{k,\ell}$.
Then
\begin{equation} \label{eq:klsum}
\left[ \UP^T\left(\sum_{s=1}^N w_s \AA{s}-P\right)\UP \right]_{k,\ell}
= \sum_{i=1}^n \sum_{j=1}^n \sum_{s=1}^N
	w_s(\AA{s}-P)_{i,j} U_{i,k} U_{j,\ell},
\end{equation}
is a sum of independent mean-zero random variables.
Since $(\AA{s}-P)_{i,j}$ is $(\nu_{s,i,j}, b_{s,i,j})$-sub-gamma,
$w_s (\AA{s}-P)_{i,j} U_{i,k} U_{j,\ell}$ is
$( w_s^2 U_{i,k}^2 U_{j,\ell}^2 \nu_{s,i,j},
	w_s |U_{i,k} U_{j,\ell}| b_{s,i,j} )$-sub-gamma
Applying a standard Bernstein inequality 
to the quantity in Equation~\eqref{eq:klsum} thus yields that
\begin{equation*}
\Pr\left[ | \UP^T(A-P)\UP |_{k,\ell} > t \right]
\le 2\exp\left\{ -\frac{ \nubar_{k,\ell} }{ \bbar_{k,\ell}^2 }
	\left(1 + \frac{\bbar_{k,\ell} t}{\nubar_{k,\ell}}
	- \sqrt{1 + \frac{2\bbar_{k,\ell} t}{ \nubar_{k,\ell} } } \right) \right\},
\end{equation*}
where
\begin{equation*}
\nubar_{k,\ell} = \sum_{s=1}^N \sum_{i=1}^n \sum_{j=1}^n
			w_s^2 U_{i,k}^2 U_{j,\ell}^2 \nu_{s,i,j},
~~~~~~
\bbar_{k,\ell} = \max_{s\in[N], i,j \in [n]}
			w_s |U_{i,k} U_{j,\ell}| b_{s,i,j}.
\end{equation*}
Taking $t = 2(\sqrt{2\nubar_{k,\ell} } + 2\bbar_{k,\ell})\log n$
is enough to ensure that
$ | \UP^T(A-P)\UP |_{k,\ell} > t$ with probability at most $2n^{-4}$.
A union bound over all $k,\ell \in [d]$ yields
\begin{equation} \label{eq:frobsquared}
\| \UP^T(A-P)\UP \|_F^2
\le 4 \sum_{k=1}^d \sum_{\ell=1}^d
	(\sqrt{2\nubar_{k,\ell}} + 2\bbar_{k,\ell})^2 \log^2 n,
\end{equation}
with probability at least $1 - 2d^2n^{-4} \ge 1 - Cn^{-2}$,
owing to our assumption that $d = O(n)$.
Taking square roots in Equation~\eqref{eq:frobsquared}
yields the first claim.

When the eigenvectors of $P$ delocalize, we have
\begin{equation*}
\nubar_{k,\ell}
= \sum_{s=1}^N \sum_{i=1}^n \sum_{j=1}^n
	w_s^2 U_{i,k}^2 U_{j,\ell}^2 \nu_{s,i,j}
\le \frac{C^2}{n^2} \sum_{s=1}^N \sum_{i=1}^n \sum_{j=1}^n
	w_s^2 \nu_{s,i,j},
\end{equation*}
and
\begin{equation*}
\bbar_{k,\ell} = \max_{s\in[N],~i,j\in[n]} w_s |U_{i,k} U_{j,\ell}| b_{s,i,j}
	\le \frac{ C \max_{s\in[N],~i,j\in[n]} w_s b_{s,i,j} }{ n },
\end{equation*}
whence we have
\begin{equation*} \begin{aligned}
\| \UP^T(A-P)\UP \|_F^2
&\le C^2 \sum_{k=1}^d \sum_{\ell=1}^d n^{-2}
	\left( 
	\sqrt{2 \sum_{s=1}^N \sum_{i=1}^n \sum_{j=1}^n
			w_s^2 \nu_{s,i,j} }
		+ 2 \max_{s,i,j} w_s b_{s,i,j} 
	\right)^2 \log^2 n \\
&\le C^2\left( 4d^2n^{-2} \sum_{s,i,j} w_s^2 \nu_{s,i,j}
	+ 8d^2n^{-2} \max_{s,i,j} w_s^2 b_{s,i,j}^2 \right) \log^2 n \\
&= \frac{ d^2C^2}{n^2} \left( \sum_{s,i,j} w_s^2 \nu_{s,i,j} + 2\max_{s,i,j} w_s^2 b_{s,i,j}^2 \right) \log^2 n,
\end{aligned} \end{equation*}
and taking square roots completes the proof.
\end{proof}

\begin{proposition} \label{prop:APtti}
Let $\{ \AA{s} \}_{s=1}^N$ be independent and for all $s=1,2,\dots,N$,
$\{ (\AA{s}-P)_{i,j} : 1 \le i \le j \le n \}$ are independent
and $(\AA{s}-P)_{i,j}$ is $(\nu_{s,i,j}, b_{s,i,j})$-sub-gamma.
Assume that $P = \E \AA{s}$ is rank $d$ with $d = O(n)$,
and let $U_{j,k}$ denote the $(j,k)$-element
of the matrix $\UP \in \R^{n \times d}$ of eigenvectors of $P$
with non-zero eigenvalues.
Define for all $i \in [n]$ and $k \in [d]$ the quantities
\begin{equation*} \label{eq:def:nubstar}
\nustar_{i,k} = \sum_{s=1}^N \sum_{j=1}^n w_s^2 U_{j,k}^2 \nu_{s,i,j},
~~~~~~
\bstar_{i,k} =  \max_{s\in[N],j \in [n]} w_s|U_{j,k}| b_{s,i,j}.
\end{equation*}
Then with probability at least $1-Cn^{-2}$,
\begin{equation*}
\| (\Atilde-P)\UP \|_{\tti} \le
	2 \max_{i\in [n]} \left(
 	\sum_{k=1}^d \left(\sqrt{2\nustar_{i,k}} + 2\bstar_{i,k}\right)^2
	\right)^{1/2} \log n
\end{equation*}
If, in addition, the top eigenvectors of $P$ delocalize,
\begin{equation*}
\| (A-P)\UP \|_{\tti} \le 
\frac{ C d^{1/2} }{ \sqrt{n} }
\max_{i \in [n]} \left( \sum_{s=1}^N \sum_{j=1}^n w_s^2 \nu_{s,i,j}
           + \max_{s\in[N], j \in [n]} w_s^2 b_{s,i,j}^2 \right)^{1/2} \log n
\end{equation*}
also with probability at least $1-Cn^{-2}$,
\end{proposition}
\begin{proof}
Fixing $i \in [n]$, the vector formed by $i$-th row of
$(\sum_{s=1}^N w_s \AA{s} -P)\UP$, which we denote by
$[(\sum_{s=1}^N w_s \AA{s} -P)\UP]_i \in \R^d$,
satisfies
\begin{equation*}
\left\| \left[ \left(\sum_{s=1}^N w_s \AA{s} -P\right) \UP \right]_i
	\right\|^2
= \sum_{k=1}^d \left( \sum_{s=1}^N \sum_{j=1}^n
		w_s (\AA{s}-P)_{i,j} U_{j,k} \right)^2.
\end{equation*}
Fix some $k \in [d]$, and consider
$Z_{i,k} = \sum_{j=1}^n \sum_{s=1}^N w_s(\AA{s}-P)_{i,j} U_{j,k}$,
which is a sum of $nN$ independent zero-mean random variables.
By our assumptions, 
$w_s U_{j,k}(\AA{s}-P)_{i,j}$ is
$(w_s^2 U_{j,k}^2 \nu_{s,i,j}, w_s|U_{j,k}| b_{s,i,j} )$-sub-gamma
and $Z_{i,k}$ is $(\nustar_{i,k}, \bstar_{i,k})$-sub-gamma.
Writing $\nustar_{i,k} = \sum_{s=1}^N \sum_{j=1}^n w_s^2 U_{j,k}^2 \nu_{s,i,j}$,
A standard Bernstein inequality 
yields
\begin{equation*}
\Pr\left[ |Z_{i,k}| > t \right]
\le 2\exp\left\{ -\left( \frac{t}{\bstar_{i,k}} + \frac{\nustar_{i,k}}{(\bstar_{i,k})^2} - \frac{\nustar_{i,k}}{(\bstar_{i,k})^2}\sqrt{1 + 2t \bstar_{i,k}/\nustar_{i,k}} \right) \right\}.
\end{equation*}
Choosing $t = 2\left(\sqrt{2\nustar_{i,k}} + 2\bstar_{i,k}\right)\log n$
suffices to ensure that the event
\begin{equation*}
\left\{ |Z_{i,k}|
	> 2\left(\sqrt{2\nustar_{i,k}} + 2\bstar_{i,k}\right)\log n \right\}
\end{equation*}
occurs with probability at most $2n^{-4}$.
A union bound over all $d$ entries of the vector
$[(\sum_{s=1}^N w_s\AA{s} -P)\UP]_i$ yields that
with probability at least $1 - 2dn^{-4}$,
\begin{equation*}
\left\| \left[\left(\sum_{s=1}^N w_s \AA{s} -P\right)\UP \right]_i
	\right\|^2
\le 4\sum_{k=1}^d
	\left(\sqrt{2\nustar_{i,k}} + 2\bstar_{i,k}\right)^2 \log^2 n,
\end{equation*}
and another union bound over all $i \in [n]$ yields that
with probability at least $1- 2dn^{-3}$,
\begin{equation} \label{eq:ttibound_loc}
\left\| \left(\sum_{s=1}^N w_s \AA{s} -P\right)\UP \right\|_{\tti}
\le 2 \max_{i \in [n]} \left(
        \sum_{k=1}^d \left(\sqrt{2\nustar_{i,k}} + 2\bstar_{i,k}\right)^2
    \right)^{1/2} \log n.
\end{equation}
Since $d = O(n)$ by assumption, this event holds with
probability at least $1 - C n^{-2}$, as we set out to prove.

If the top eigenvectors of $P$ delocalize, then we have
\begin{equation*}
\nustar_{i,k} \le \frac{C^2}{n} \sum_{s=1}^N \sum_{j=1}^n w_s^2 \nu_{s,i,j},
~~~~~~
  \bstar_{i,k} \le \frac{C\max_{s\in[N], j \in [n]} w_s b_{s,i,j} }{\sqrt{n}},
\end{equation*}
and we can strengthen Equation~\eqref{eq:ttibound_loc} to
\begin{equation*}
\| (A-P)\UP \|_{\tti} \le 
\frac{ C d^{1/2} }{ \sqrt{n} }
 \max_{i \in [n]} \left( \sum_{s=1}^N \sum_{j=1}^n w_s^2 \nu_{s,i,j}
		+ \max_{s\in[N], j \in [n]} w_s^2 b_{s,i,j}^2 \right)^{1/2},
\end{equation*}
also with probability at least $1-Cn^{-2}$, again since $d=O(n)$
by assumption.
\end{proof}

We are now ready to prove Theorem~\ref{thm:subexpXtilde}.

\begin{delayedproof}{Proof of Theorem~\ref{thm:subexpXtilde}}
By Lemma~\ref{lem:mxspecbound}, with probability at least $1-Cn^{-2}$,
\begin{equation} \label{eq:tildespecbound}
\| \Atilde - P \|
\le
15 \sqrt{ n \sum_{s=1}^N w_s^2 (\sqrt{2 \nu_s} + b_s)^2}\log n
\le C \sqrt{n \sum_{s=1}^N w_s^2 (\nu_s + b_s^2) } \log n.
\end{equation}
The assumption in Equation~\eqref{eq:paragrowth} ensures that for
suitably large $n$, $\| \Atilde - P \| \le \lambda_d(P)/2$,
so that Equation~\eqref{eq:assum:specbound} of Lemma~\ref{lem:genericbound}
is satisfied with high probability.
The result will now follow immediately from Lemma~\ref{lem:genericbound},
provided we can bound both
$\| \UP^T(\Atilde-P)\UP \|_F$ and $\|(\Atilde-P)\UP\|_{\tti}$,
also with high probability.

Proposition~\ref{prop:UTAPU} implies that with probability
at least $1-Cn^{-2}$,
\begin{equation*}
\| \UP^T( \Atilde - P)\UP \|_F
\le
C d \left( \sum_{s=1}^N w_s^2( \nu_s + b_s^2 ) \right)^{1/2} \log n,
\end{equation*}
and Proposition~\ref{prop:APtti} implies that with probability
at least $1-Cn^{-2}$,
\begin{equation*}
\| (\Atilde-P)\UP \|_{\tti}
\le
C \sqrt{d} \left( \sum_{s=1}^N w_s^2(\nu_s + b_s^2) \right)^{1/2} \log n.
\end{equation*}
Plugging the above two bounds, along with Equation~\eqref{eq:tildespecbound},
into Lemma~\ref{lem:genericbound}, since $d \ge 1$, we have
\begin{equation*} \begin{aligned}
\| \UAtilde \SAtilde^{1/2} &- \UP \SP^{1/2} W \|_{\tti} \\
&\le
\frac{ Cd }{\lambda_d^{1/2}(P)}
\left( \sum_{s=1}^N w_s^2(\nu_s + b_s^2) \right)^{1/2} \log n
+
\frac{ Cd \kappa(P) n}{ \lambda_d^{3/2}(P) }
\left( \sum_{s=1}^N w_s^2(\nu_s + b_s^2) \right) \log^2 n,
\end{aligned} \end{equation*}
which completes the proof.
\end{delayedproof}

\subsection{Proof of Theorem~\ref{thm:minimax} (Minimax estimation in spectral norm)}
\label{apx:minimax}

In this section, we establish Theorem~\ref{thm:minimax},
showing that the weighted network average $\sum_s \wopt_s \AA{s}$
recovers $P$ at the minimax rate when the edge errors in each network are
drawn from a network-specific normal distribution.
Our main tool is Theorem 2.7 in \cite{Tsybakov2009}, which we restate here,
adapted to the present setting.

\begin{theorem} \label{thm:tsybakov}
{\em (\cite{Tsybakov2009} Theorem 2.7)}
Let $\Theta$ be a parameter space endowed with a distance $\delta(\cdot,\cdot)$
and containing $M \ge 1$ elements
$\theta_0,\theta_1,\dots,\theta_M$, such that, for $s > 0$,
\begin{enumerate}
\item[i] For all $0 \le i < j \le M$,  $\delta(\theta_i, \theta_j) \ge 2s$
		\label{item:tsybakov:distlb}
\item[ii] Writing $P_j$ for the distribution induced by $\theta_j$,
	\label{item:tsybakov:KL}
	$P_{j} \ll P_0$ for all $j = 1,2,\dots,M$, and
	\begin{equation*}
	\frac{1}{M} \sum_{j=1}^M \KL( P_j, P_0 ) \le \alpha \log M
	\end{equation*}
	for some $\alpha \in (0,1/8)$.
\end{enumerate}
Then
\begin{equation*}
\inf_{\thetahat} \sup_{\theta \in \Theta} \E \delta( \thetahat, \theta )
\ge c_\alpha s,
\end{equation*}
where the infimum is over all estimators and $c_\alpha$ is a constant
depending only on $\alpha$.
\end{theorem}

\begin{delayedproof}{Proof of Theorem~\ref{thm:minimax}}
We will apply Theorem~\ref{thm:tsybakov} to a suitably-chosen collection
of symmetric $n$-by-$n$ matrices, under the distance
\begin{equation*}
\delta( P^{(1)}, P^{(2)} ) = \| P^{(1)} - P^{(2)} \|,
\end{equation*}
where $P^{(1)}, P^{(2)} \in \R^{n \times n}$.
Letting $\symmetricmxs_n$ denote the set of symmetric $n$-by-$n$ matrices,
any $P \in \symmetricmxs_n$ induces a distribution on
$(\R^{n \times n})^N$ given by
\begin{equation*}
\AA{s}_{ij} \indsim \calN( P_{ij}, \rho_s ),
~~~1 \le i \le j \le n, s=1,2,\dots,N,
\end{equation*}
Making use of the independence structure,
the Kullback-Leibler divergence between the distributions induced by two
symmetric matrices $P^{(1)}, P^{(2)} \in \symmetricmxs_n$ is given by
\begin{equation} \label{eq:kl:indep}
\KL( P^{(1)}, P^{(2)} )
= \sum_{1 \le i \le j \le n} \sum_{s=1}^N 
	\KL\left( \calN( P^{(1)}_{ij}, \rho_s ) ,
		\calN( P^{(2)}_{ij}, \rho_s ) \right).
\end{equation}
The KL-divergence betwen two normal distributions $Q_1,Q_2$ with 
means $\mu_1$ and $\mu_2$ and common variance $\sigma^2 > 0$ is given by
(letting $Z_1 \sim Q_1$)
\begin{equation*}
\KL\left( Q_1, Q_2 \right)
= \E Q_1(Z_1) \log \frac{ Q_1(Z_1) }{ Q_2( Z_1) }
= \frac{ (\mu_1 - \mu_2)^2 }{ \sigma^2 }.
\end{equation*}
Applying this identity to Equation~\eqref{eq:kl:indep},
\begin{equation} \label{eq:KL:frob}
\KL( P^{(1)}, P^{(2)} ) 
= \sum_{1 \le i \le j \le n} \sum_{s=1}^N
	\frac{ \left(P^{(1)}_{ij} - P^{(2)}_{ij} \right)^2 }{ \rho_s }
\le \| P^{(1)} - P^{(2)} \|_F^2 \sum_{s=1}^N \rho_s^{-1}.
\end{equation}

To each $\xi \in \{0,1\}^n$, associate a symmetric matrix
$M^{(\xi)} \in \R^{n \times n}$ given by
\begin{equation*}
M^{(\xi)}_{ij} = \begin{cases}
		\xi_j &\mbox{ if } 1 = i \le j \le n, \xi_j=1 \\
		0 &\mbox{ if } 1 < i \le j \le n \\
		M^{(\xi)}_{ji} &\mbox{ if } 1 \le j < i \le n \\
		  \end{cases}
\end{equation*}
Let $\calM_{n,d} \subseteq \{0,1\}^n$ be a collection such that for all
$\xi,\xi' \in \calM_{n,d}$ with $\xi \neq \xi'$,
the Hamming distance $d_H(\xi,\xi')$ between $\xi$ and $\xi'$ is at least $d$,
for some $d \le n$ to be specified below.
Letting $\eta > 0$, for each $\xi \in \calM_{n,d}$, construct
\begin{equation*}
  P^{(\xi)} = \frac{ \eta }{ \sqrt{d} } M^{(\xi)}.
\end{equation*}
Recall that for a matrix $P \in \R^{n \times n}$,
the spectral norm is lower-bounded by
the maximum Euclidean norm of any column of $P$.
That is, letting $e_1,e_2,\dots,e_n \in \R^n$
denote the standard basis vectors,
\begin{equation*}
\| P \|
\ge \max_{i \in [n]} \| P e_i \|
= \max_{i \in [n]} \| P_{\cdot,i} \|.
\end{equation*}
Then for any $\xi, \xi' \in \calM_{n,d}$ distinct,
\begin{equation*}
\left\| P^{(\xi)} - P^{(\xi')} \right\|^2
\ge \sum_{j=1}^n \left( P^{(\xi)}_{j,1} - P^{(\xi')}_{j,1} \right)^2
= \frac{ \eta^2 d_H( \xi, \xi' ) }{ d }.
\end{equation*}
By definition of $\calM_{n,d}$, $d_H( \xi, \xi' ) \ge d$, whence
for all $\xi, \xi' \in \calM_{n,d}$ distinct,
\begin{equation} \label{eq:spec:LB}
\left\| P^{(\xi)} - P^{(\xi')} \right\| \ge \eta.
\end{equation}

In addition to the matrices $\{ P^{(\xi)} : \xi \in \{0,1\}^n \}$,
define the matrix $P^{(0)} \in \R^{n \times n}$ by
\begin{equation*}
P^{(0)}_{ij}
= \begin{cases}
	\eta &\mbox{ if } i=j=n \\
	0 &\mbox{ otherwise. }
	\end{cases}
\end{equation*}
By construction, for all $\xi \in \calM_{n,d}$,
\begin{equation*}
\| P^{(\xi)} - P^{(0)} \|
\ge \eta \left( 1 + \frac{1}{d} \right)^{1/2}
\ge \eta.
\end{equation*}
Thus, the collection of matrices
$\left\{ P^{(t)} : t \in \calM_{n,d} \cup \{0\} \right\}$
obeys Condition~\ref{item:tsybakov:distlb} in Theorem~\ref{thm:tsybakov}
with $s = \eta/2$.

Similarly, for all $\xi \in \calM_{n,d}$,
writing $| \xi |$ to denote the number of non-zero entries in $\xi$,
\begin{equation} \label{eq:xi0:frobbound}
\left\| P^{(\xi)} - P^{(0)} \right\|_F^2
= \eta^2 + \left( P^{(\xi)}_{1,1} \right)^2
	+ 2 \sum_{j=2}^n \left( P^{(\xi)}_{1,j} \right)^2
\le \eta^2 + \frac{ 2| \xi | \eta^2 }{ d }
\le \eta^2 \left( 1 + \frac{2n}{d} \right),
\end{equation}
where we have used the trivial upper bound $|\xi| \le n$.
Applying Equation~\eqref{eq:xi0:frobbound} to Equation~\eqref{eq:KL:frob},
it follows that
\begin{equation} \label{eq:nextone}
\frac{1}{| \calM_{n,d} |} \sum_{\xi \in \calM_{n,d} }
	\KL\left( P^{(\xi)}, P^{(0)} \right)
\le \eta^2\left( 1 + \frac{2n}{d} \right)
	\left( \sum_{s=1}^N \rho_s^{-1} \right).
\end{equation}
By the Gilbert-Varshamov bound
\citep[see, e.g.,][and citations therein]{JiaVar2004}
we can choose $\calM_{n,d}$ so that
\begin{equation*}
\left| \calM_{n,d} \right|
\ge \frac{ 2^n }{ \sum_{j=0}^{d - 1} \binom{n}{j} }.
\end{equation*}
Letting $H(p) = -p \log_2 p - (1-p) \log_2(1-p) \in [0,1]$
denote the binary entropy function, recall the inequality
\citep[][Chapter 4]{Ash1965}
$\sum_{j=0}^{d-1} \binom{n}{j}
\le 2^{n H\left( \frac{d-1}{n} \right) }$, valid for $d-1 < n/2$.
Taking logarithms, we have
\begin{equation} \label{eq:logS:LB}
\log \left| \calM_{n,d} \right|
\ge \log \frac{ 2^n }{ 2^{n H( (d-1)/n) } }
= n \left( 1 - H\left( \frac{d-1}{n} \right) \right) \log 2.
\end{equation}
Choosing $d=\lfloor n/3 \rfloor + 1$ is enough to ensure that
for all $n \ge 2$
$n/d$ is bounded by a constant while $H((d-1)/n)$ is bounded away from $1$,
so that
\begin{equation*} 
\left( 1 + \frac{2n}{d} \right)
\le C \left( 1 - H\left( \frac{d-1}{n} \right) \right) \log 2
\end{equation*}
for some constant $C > 0$ not depending on $n$.
Applying this bound to Equation~\eqref{eq:nextone},
\begin{equation} \label{eq:KL:sumbound}
\frac{1}{| \calM_{n,d} |} \sum_{\xi \in \calM_{n,d} }
	\KL\left( P^{(\xi)}, P^{(0)} \right)
\le C \left( 1 - H\left( \frac{1}{c_d} \right) \right)
	\eta^2 \left( \sum_{s=1}^N \rho_s^{-1} \right).
\end{equation}
Combining Equation~\eqref{eq:logS:LB} with Equation~\eqref{eq:KL:sumbound},
\begin{equation} \label{eq:KL:lastbound}
\frac{1}{| \calM_{n,d} |} \sum_{\xi \in \calM_{n,d} }
	\KL\left( P^{(\xi)}, P^{(0)} \right)
\le
\frac{ C \eta^2 }{ n }\left( \sum_{s=1}^N \rho_s^{-1} \right)
\log \left| \calM_n \right|.
\end{equation}
Choosing
\begin{equation*}
\eta = C_1 \sqrt{n} \left( \sum_{s=1}^N \rho_s^{-1} \right)^{-1/2}
\end{equation*}
for suitably small constant $C_1 > 0$ ensures that 
\begin{equation*} 
C \eta^2 \le \alpha n \left( \sum_{s=1}^N \rho_s^{-1} \right)^{-1}
\end{equation*}
for some $\alpha \in (0,1/8)$.
Plugging this into Equation~\eqref{eq:KL:lastbound}, it follows that
\begin{equation*}
\frac{1}{| \calM_{n,d} |} \sum_{\xi \in \calM_{n,d} }
        \KL\left( P^{(\xi)}, P^{(0)} \right)
\le \alpha \log \left| \calM_n \right|,
\end{equation*}
and Condition~\ref{item:tsybakov:KL} of Theorem~\ref{thm:tsybakov}
is satisfied with $s/2 = \eta$.
Applying this theorem, we conclude that
\begin{equation*}
\inf_{\Phat} \sup_{P \in \symmetricmxs_n}
	\E \| \Phat - P \|
\ge C \sqrt{n} \left( \sum_{s=1}^N \rho_s^{-1} \right)^{-1/2},
\end{equation*}
completing the proof.
\end{delayedproof}

\subsection{Proof of Theorem~\ref{thm:subgamma:XhatXopt} (Estimating the sub-gamma weights)}
\label{apx:estimation}
In this section, we provide a proof of Theorem~\ref{thm:subgamma:XhatXopt}.
We remind the reader that in this section,
$\{ \wopt_s \}_{s=1}^N$ are non-negative weights, summing to $1$,
defined by
\begin{equation} \label{eq:def:wopt:reminder}
\wopt_s = \frac{ \nub{s}^{-1} }{ \sum_{t=1}^N \nub{t}^{-1} }
\end{equation}
for each $s \in [N]$, where $(\nu_s,b_s)$ is the sub-gamma parameter
for the edges in the $s$-th network.
We further remind the reader that we estimate these weights by
\begin{equation} \label{eq:def:reminder:what}
\what_s = \frac{ \rhohat_s^{-1} }{ \sum_{t=1}^n \rhohat_t^{-1} },
\end{equation}
where for each $s \in [N]$, letting
$\Phat^{(s)}$ denote the rank-$d$ eigenvalue truncation of $\AA{s}$,
\begin{equation} \label{eq:def:reminder:rhohat}
\rhohat_s = \frac{ \sum_{1 \le i \le j \le n } (\AA{s} - \Phat^{(s)})_{i,j}^2 }
                {16 n(n+1) }.
\end{equation}
For ease of notation, define
\begin{equation} \label{eq:def:rhotilde}
\rhotilde_s = \frac{ \sum_{1 \le i \le j \le n } (\AA{s} - P)_{i,j}^2 }{16 n(n+1) },
\end{equation}
noting that by definition of $\tau_s$ in Equation~\eqref{eq:def:tau},
it follows from a basic property of sub-gamma random variables that
\begin{equation} \label{eq:def:reminder:tau}
\tau_s = \E \rhotilde_s \le \frac{8\nu_s + 32b_s^2}{32} \le \nu_s + b_s^2.
\end{equation}

The following technical results will prove useful in our main proof.

\begin{proposition} \label{prop:APfrob}
Suppose the networks $\AA{1},\AA{2},\dots,\AA{N}$
are independent and for each $s=1,2,\dots,N$ the edges
$\{ (\AA{s} - P)_{i,j} : 1 \le i \le j \le n \}$ are independent
$(\nu_s,b_s)$-sub-gamma random variables.
With probability at least $1-Cn^{-2}$ it holds for all $s \in [N]$ that
\begin{equation*}
\| \AA{s} - P \|_F \le C n \nub{s}^{1/2} (\log N + \log n) .
\end{equation*}
\end{proposition}
\begin{proof}
By a standard tail inequality for sub-gamma random variables, for all $t \ge 0$,
\begin{equation*} 
\Pr\left[ (\AA{s} - P)_{i,j}^2 > (\sqrt{2\nu_s t} + b_s t)^2 \right]
= \Pr\left[ |\AA{s}_{i,j} - P_{i,j}| > \sqrt{2\nu_s t} + b_s t \right]
  \le C \exp(-t).
\end{equation*}
Setting $t = 4\log n + \log N$, taking a union bound over all
$\{(i,j) : 1 \le i \le j \le n\}$
and upper bounding $(\sqrt{2\nu_s t} + b_s t)^2 \le Ct^2\nub{s}$
since $t \ge 1$,
it holds with probability at least $1-CN^{-1}n^{-2}$ that
for all $1 \le i \le j \le n$,
\begin{equation*}
(\AA{s}-P)_{i,j}^2 \le C\nub{s}(\log N + \log n)^2.
\end{equation*}
Summing over $1 \le i < j \le n$,
$ \| \AA{s} - P \|_F^2 \le C n^2 \nub{s}(\log N + \log n)^2$
with probability at least $1-CN^{-1} n^{-2}$.
A union bound over $s \in [N]$ yields the result.
\end{proof}

\begin{proposition} \label{prop:rhohatrhotilde}
Suppose that the networks $\AA{1},\AA{2},\dots,\AA{N}$
are independent and for each $s=1,2,\dots,N$ the edges
$\{ (\AA{s} - P)_{i,j} : 1 \le i \le j \le n \}$ are independent
$(\nu_s,b_s)$-sub-gamma random variables.
Let $\rhohat_s$ and $\rhotilde_s$ be as defined in
Equations~\eqref{eq:def:reminder:rhohat} and~\eqref{eq:def:rhotilde},
respectively, for all $s \in [N]$,
and suppose that $d \le n$ for all suitably large $n$.
With probability at least $1-Cn^{-2}$, it holds for all $s \in [N]$ that
\begin{equation*}
| \rhohat_s - \rhotilde_s |
\le \frac{C \sqrt{d}\nub{s}(\log N + \log n)^2}{\sqrt{n}}.
\end{equation*}
\end{proposition}
\begin{proof}
By definition,
\begin{equation*} \begin{aligned}
| \rhohat_s - \rhotilde_s |
&= \frac{ \left| \sum_{1 \le i \le j \le n} (\AA{s}-\Phat^{(s)})_{i,j}^2
				-(\AA{s}-P)_{i,j}^2 \right| }{16 n(n+1)} \\
&\le \frac{ \sum_{1 \le i \le j \le n} (\Phat^{(s)}-P)_{i,j}^2 }
		{16 n(n+1) }
	+ \frac{
	\sum_{1 \le i \le j \le n}
	\left|(\AA{s}-P)_{i,j}(\Phat^{(s)}-P)_{i,j} \right| }
	{16n(n+1)}  \\
&\le \frac{ C \| \Phat^{(s)} - P \|_F^2 }{ n^2 }
	+ \frac{ C \| \AA{s}-P \|_F \| \Phat^{(s)}-P \|_F }{ n^2 },
\end{aligned} \end{equation*}
where the second bound follows from an application of the
Cauchy-Schwarz inequality.
Since $\Phat^{(s)}$ is the rank-$d$ truncation of $\AA{s}$,
and both $\Phat^{(s)}$ and $P$ are rank-$d$,
\begin{equation*}
\| \Phat^{(s)} - P \|_F^2 \le Cd \| \AA{s} - P \|^2
\le C d n \nub{s} (\log N + \log n)^2,
\end{equation*}
where the second inequality holds with high probability
simultaneously for all $s \in [N]$
by applying Lemma~\ref{lem:mxspecbound} to
$\| \AA{s} - P\|$ separately for each $s \in [N]$
followed by a union bound argument similar to that given in the previous proof.
Taking square roots in the above display and applying
Proposition~\ref{prop:APfrob} to bound $\|\AA{s}-P\|_F$, we conclude that
\begin{equation*} \begin{aligned}
| \rhohat_s - \rhotilde_s |
&\le \frac{ C \| \Phat^{(s)} - P \|_F^2 }{ n^2 }
        + \frac{ C \| \AA{s}-P \|_F \| \Phat^{(s)}-P \|_F }{ n^2 } \\
&\le \frac{ C d \nub{s} (\log N + \log n)^2 }{ n }
+ \frac{ C \sqrt{d} \nub{s} ( \log N + \log n)^2 }{ \sqrt{n} } \\
&\le C \nub{s} \frac{\sqrt{d} (\log N + \log n)}{\sqrt{n}},
\end{aligned} \end{equation*}
where we have used the assumption that
$d/n \le 1$ for all suitably large $n$.
\end{proof}

\begin{proposition} \label{prop:rhotildetau}
Suppose that the networks $\AA{1},\AA{2},\dots,\AA{N}$
are independent and that for each $s=1,2,\dots,N$, the edges
$\{ (\AA{s} - P)_{i,j} : 1 \le i \le j \le n \}$ are independent
$(\nu_s,b_s)$-sub-gamma random variables.
Let $\rhohat_s$ and $\rhotilde_s$ be as defined in
Equations~\eqref{eq:def:reminder:rhohat} and~\eqref{eq:def:rhotilde},
respectively, for all $s \in [N]$
and suppose that the growth assumption in Equation~\eqref{eq:updatedgrowth} 
from Theorem~\ref{thm:subgamma:XhatXopt} holds.
That is, there exists a positive integer $k$ such that
\begin{equation} \label{eq:updatedgrowth2}
 \frac{ n^{k-2} d^k \left( \log N + \log n \right)^{4k} }{N} = \Omega( 1 ).
\end{equation}
Then with probability $1-O(n^{-2})$,
it holds for all $s \in [N]$ that
\begin{equation*}
|\rhotilde_s - \tau_s|
\le \frac{\sqrt{d} \nub{s}(\log N + \log n)^2}{\sqrt{n}}.
\end{equation*}
\end{proposition}
\begin{proof}
By definition,
\begin{equation*}
\rhotilde_s - \tau_s
= \sum_{1 \le i \le j \le n} \frac{(\AA{s}-P)_{i,j}^2 - \E(\AA{s}-P)_{i,j}^2}{16n(n+1)},
\end{equation*}
is a sum of independent mean-$0$ random variables.
By Chebyshev's inequality,
\begin{equation} \label{eq:counting:chebyshev}
\Pr\left[ | \rhotilde_s - \tau_s| > t \right] \le 
\frac{ C \E \left[ \sum_{1 \le i \le j \le n} (\AA{s}-P)_{i,j}^2 - \E(\AA{s}-P)_{i,j}^2
		\right]^\ell }{ n^{2\ell} t^\ell }
\end{equation}
for any $t > 0$ and any even integer $\ell \ge 2$.
We will bound the expectation on the right-hand side via a standard counting
argument.

For ease of notation, let
$ Z_{i,j} = (\AA{s}-P)_{i,j}^2 - \E(\AA{s}-P)_{i,j}^2,$ so that
\begin{equation*}
\rhotilde_s - \E \rhotilde_s
= \frac{ \sum_{1\le i \le j \le n} Z_{i,j} }{ 16n(n+1) }.
\end{equation*}
The $\{Z_{i,j} : 1 \le i \le j \le n\}$ are independent mean-$0$,
and since $(\AA{s}-P)_{i,j}$ are $(\nu_s,b_s)$-sub-gamma,
all moments of $Z_{i,j}$ exist, with
\begin{equation*}
\E Z_{i,j}^k \le \E(\AA{s}-P)_{i,j}^{2k}
	\le C_k \nub{s}^k,
\end{equation*}
where $C_k$ is a constant depending on $k$ but not on any other parameters.
Let $\ivec$ denote an $\ell$-tuple of numbers from $[n]$,
$\ivec = (i_1,i_2,\dots,i_\ell)$ with $i_a \in [n]$ for all $a \in [\ell]$.
Write $\ivec \le \jvec$ to mean that
$i_a \le j_a$ for all $a=1,2,\dots,\ell$.
Defining the set
$\calI_{n,\ell} = \{ (\ivec,\jvec) : \ivec \le \jvec \}$, we have
\begin{equation*}
\E \left| \sum_{1 \le i \le j \le n} Z_{i,j} \right|^\ell
= \sum_{(\ivec,\jvec) \in \calI_{n,\ell}} \E \prod_{k=1}^\ell Z_{i_k,j_k}.
\end{equation*}
We can identify each $(\ivec,\jvec) \in \calI_{n,\ell}$ with
an $\ell$-tuple of pairs
$\left((i_1,j_1),(i_2,j_2),\dots,(i_\ell,j_\ell)\right)$.
For $1 \le i \le j \le n$,
let $m_{(i,j)}(\ivec,\jvec)$ denote the number of times that
the pair $(i,j)$ appears in the $\ell$-tuple of pairs $(\ivec,\jvec)$.
That is,
\begin{equation*}
m_{(i,j)}(\ivec,\jvec) = |\{ a \in [\ell]: i_a = i, j_a=j \}|.
\end{equation*}
We can rewrite our expectation of interest as
\begin{equation} \label{eq:expecsum}
\E \left| \sum_{1 \le i \le j \le n} Z_{i,j} \right|^\ell
= \sum_{(\ivec,\jvec) \in \calI_{n,\ell}}
  \E \prod_{(i,j) \in (\ivec,\jvec)}  Z_{i,j}^{m_{(i,j)}(\ivec,\jvec)},
\end{equation}
and we see immediately that by independence of the $Z_{i,j}$,
if $m_{(i_a,j_a)}(\ivec,\jvec) = 1$ for some
$a \in [\ell]$, then the corresponding term in the sum has
\begin{equation*}
\E \prod_{k=1}^\ell Z_{i_k,j_k}
= \E \prod_{(i,j) \in S(\ivec,\jvec)}  Z_{i,j}^{m_{(i,j)}(\ivec,\jvec)}
= 0.
\end{equation*}
Since $\{ Z_{i,j} : 1 \le i \le j \le n\}$ are independent
$(\nu_s,b_s)$-sub-gamma, H\"older's inequality implies that
for any $(\ivec,\jvec) \in \calI_{n,\ell}$,
letting $E(\ivec,\jvec) = \{ (i_a,j_a) : a \in [\ell] \}$,
\begin{equation} \label{eq:holder}
\E \prod_{k=1}^\ell Z_{i_k,j_k}
= \prod_{e \in E(\ivec,\jvec)} \E Z_e^{m_e(\ivec,\jvec)}
\le \prod_{e \in E(\ivec,\jvec)}
	\left( \E Z_e^\ell \right)^{\frac{m_e(\ivec,\jvec)}{\ell}}
\le C_\ell\nub{s}^\ell .
\end{equation}

We identify each $(\ivec,\jvec) \in \calI_{n,\ell}$ with a partition of $[\ell]$
in the following way: for $a,b \in [\ell]$,
take $a \sim b$ if and only if $(i_a,j_a) = (i_b,j_b)$.
Under this identification,
the nonzero elements of the sum on the right-hand side of
Equation~\eqref{eq:expecsum} are precisely those whose corresponding partition
of $[\ell]$ has no singleton parts.
Thus, to bound the expectation on the left-hand side of
Equation~\eqref{eq:expecsum}, it will suffice to count how many such
partitions correspond to non-zero terms in the right-hand sum,
and apply the bound in Equation~\eqref{eq:holder} to those terms.

Let $\calP_\ell^+$ denote the set of all partitions of $[\ell]$ having
no singleton part.
Any $\pi \in \calP_\ell^+$ can correspond to at most
$(n(n+1)/2)^{\ell/2}$ pairs $(\ivec,\jvec) \in \calI_{n,\ell}$,
since we must associate each part of $\pi$ with
some $(i,j) \in [n]$ satisfying $i \le j$,
and each $\pi \in \calP_\ell^+$ has at most $\ell/2$ parts.
Thus, we can can bound
\begin{equation*}
\E \left| \sum_{1 \le i \le j \le n} Z_{i,j} \right|^\ell
= \sum_{\pi \in \calP_\ell^+} 
        \E \prod_{(i,j) \in (\ivec,\jvec)}  Z_{i,j}^{m_{(i,j)}(\ivec,\jvec)}
\le \frac{ C_\ell |\calP_\ell^+| n^\ell \nub{s}^\ell }{ 2^{\ell/2} }
\le C_\ell n^{\ell} \nub{s}^\ell,
\end{equation*}
where we have used the fact that $|\calP_\ell^+| \le 2^\ell$
and we have gathered all constants,
possibly depending on $\ell$ but not on $n$, into $C_\ell$.
Plugging this back into
Equation~\eqref{eq:counting:chebyshev}, we conclude that
for each $s \in [N]$,
\begin{equation*}
\Pr\left[ \left|\rhotilde_s - \tau_s \right| > t_s \right]
\le \frac{ \E \left|\rhotilde_s - \tau_s \right|^\ell }{ t_s^\ell }
\le \frac{ C n^\ell \nub{s}^\ell }{ n^{2\ell} t_s^\ell }
= \frac{ C \nub{s}^\ell } { n^\ell t_s^\ell} .
\end{equation*}


Setting 
\begin{equation*}
t_s = \frac{\sqrt{d} (\nu_s + b_s^2) (\log N + \log n)^2 }{\sqrt{n}},
\end{equation*}
we have
\begin{equation*}
\Pr\left[ \left|\rhotilde_s - \tau_s \right| > t_s \right]
\le
\frac{ C \nub{s}^\ell } { n^\ell t_s^\ell}
= \frac{C}{(nd)^{\ell/2} (\log N + \log n)^{2\ell}},
\end{equation*}
and a union bound over $s \in [N]$ yields
\begin{equation*}
\Pr\left[ \exists s \in [N] : \left|\rhotilde_s - \tau_s \right| > t_s \right]
\le
\frac{CN}{(nd)^{\ell/2} (\log N + \log n)^{2\ell}}.
\end{equation*}
Taking $\ell = 2k$ for a positive integer $k$ chosen in accordance
with our growth assumption in Equation~\eqref{eq:updatedgrowth2},
it follows that
\begin{equation*}
\Pr\left[ \exists s \in [N] : \left|\rhotilde_s - \tau_s \right| > t_s \right]
= O(n^{-2}),
\end{equation*}
which completes the proof.
\end{proof}


\begin{proposition} \label{prop:whatu}
Let $\AA{1},\AA{2},\dots,\AA{N}$
be independent random networks with shared expectation
$P = X X^T \in \R^{n \times n}$, where $X \in \R^{n \times d}$,
and suppose that for all $s \in [N]$,
$\{ (\AA{s}-P)_{i,j} : 1 \le i \le j \le n \}$
are independent $(\nu_s,b_s)$-sub-gamma random variables.
Let $\{ \what_s \}_{s=1}^N$ and $\{ \tau_s \}_{s=1}^N$ be as in
Equations~\eqref{eq:def:reminder:what} and~\eqref{eq:def:reminder:tau},
respectively, and define
\begin{equation} \label{eq:def:u}
u_s = \frac{ \tau_s^{-1} }{ \sum_{t=1}^N \tau_t^{-1} }
\end{equation}
for each $s \in [N]$.
Under the same growth assumptions as Theorem~\ref{thm:subgamma:XhatXopt},
for all suitably large $n$, it holds with probability $1-O(n^{-2})$ that
\begin{equation*}
\sum_{s=1}^N |\what_s-u_s| \nub{s}^{1/2}
\le \frac{ C \log n }{ \log n + \log N }
\left( \sum_{s=1}^N \nub{s}^{-1} \right)^{-1/2}.
\end{equation*}
\end{proposition}
\begin{proof}
Applying the triangle inequality followed by
Propositions~\ref{prop:rhohatrhotilde} and~\ref{prop:rhotildetau},
with probability $1- O(n^{-2})$
it holds for all $s \in [N]$ that
\begin{equation} \label{eq:estrhohattau}
|\rhohat_s - \tau_s|
\le | \rhohat_s - \rhotilde_s| + | \rhotilde_s - \tau_s |  
\le \frac{C\sqrt{d} \nub{s} (\log N + \log n)^2}{\sqrt{n}}
 = C \nub{s}\gamma_n,
\end{equation}
where, for ease of notation, we let
\begin{equation*} 
\gamma_n = \frac{\sqrt{d} (\log N + \log n)^2}{\sqrt{n}}.
\end{equation*}

Thus, for any $s\in [N]$, and $n$ suitably large,
\begin{equation} \label{eq:rhohatinv2tauinv} \begin{aligned}
|\rhohat_s^{-1} - \tau_s^{-1}|
&= \frac{ | \rhohat_s - \tau_s | }{ \tau_s \rhohat_s }
\le  \frac{ C \nub{s} \gamma_n }{ \tau_s \rhohat_s }
\le \frac{ C  \nub{s} \gamma_n }
	{ \tau_s^2 \big(1 -  \tau_s^{-1}|\rhohat_s  -\tau_s| \big) }\\
&\le \frac{ C \nub{s} \gamma_n }
	{ \tau_s^2 \big(1 -  C\tau_s^{-1} \nub{s} \gamma_n \big)} ,
\end{aligned} \end{equation}
where the inequalities follow from successive application
of Equation~\eqref{eq:estrhohattau}.

Expanding the definitions of $\what_s$ and $u_s$,
\begin{equation*} \begin{aligned}
\sum_{s=1}^N |\what_s-u_s| & \nub{s}^{1/2}
=
\sum_{s=1}^N
	\left| \frac{\rhohat_s^{-1} }{ \sum_t \rhohat_t^{-1} }
	-\frac{\tau_s^{-1} }{ \sum_t \tau_t^{-1} } \right| \nub{s}^{1/2} \\
&\le
\sum_{s=1}^N \frac{ \left| \rhohat_s^{-1} - \tau_s^{-1} \right| \nub{s}^{1/2} }
		{ \sum_t \tau_t^{-1} }
+ \sum_{s=1}^N \frac{ \rhohat_s^{-1} \nub{s}^{1/2}
		 \sum_t \left| \rhohat_t^{-1} - \tau_t^{-1} \right| }
		{ \left( \sum_t \rhohat_t^{-1} \right)
			\left( \sum_t \tau_t^{-1} \right) }.
\end{aligned} \end{equation*}
Applying the bound in Equation~\eqref{eq:rhohatinv2tauinv}
and using our assumption
in Equation~\eqref{eq:assum:ratiogrowth}
that $\tau_s^{-1} \nub{s} \gamma_n = o(1)$
uniformly over $s \in [N]$,
\begin{equation} \label{eq:whatu:bound1}
\begin{aligned}
\sum_{s=1}^N & |\what_s-u_s| \nub{s}^{1/2} \\
&\le
\sum_{s=1}^N \frac{ C \gamma_n \tau_s^{-2} \nub{s} \nub{s}^{1/2} }
		{ \sum_t \tau_t^{-1} }
    + \sum_{s=1}^N \sum_{t=1}^N
	\frac{ C \gamma_n \tau_t^{-2} \nub{t} \tau_s^{-1} \nub{s}^{1/2} }
		{ \left( \sum_t \rhohat_t^{-1} \right)
			\left( \sum_t \tau_t^{-1} \right) }.
\end{aligned} \end{equation}
Considering the first of these two sums, we have
\begin{equation} \label{eq:whatu:sum1:a}
\sum_{s=1}^N \frac{ \gamma_n \tau_s^{-1} \nub{s}^{3/2} }
                { \tau_s \sum_t \tau_t^{-1} }
\le \left( \gamma_n \max_s \tau_s^{-1} \nub{s} \right)
	\sum_{s=1}^N \frac{ \tau_s^{-1} \nub{s}^{1/2} }
                	{ \sum_t \tau_t^{-1} }.
\end{equation}
By concavity of the square root function,
\begin{equation} \label{eq:concavpayoff}
\sum_{s=1}^N \frac{ \tau_s^{-1} \sqrt{ \nub{s} \sum_t \nub{t}^{-1} } }
                { \sum_t \tau_t^{-1} }
\le
\left( \sum_{s=1}^N \frac{ \tau_s^{-1} \sum_t \nub{t}^{-1} }
			{ \nub{s}^{-1} \sum_t \tau_t^{-1} } \right)^{1/2}
= \sqrt{ \sum_{s=1}^N \frac{ u_s }{ \wopt_s } },
\end{equation}
where we have used the definitions of $\{ u_s \}_{s=1}^N$
and $\{ \wopt_s \}_{s=1}^N$ in Equations~\eqref{eq:def:u}
and~\eqref{eq:def:wopt:reminder}, respectively.
Rearranging and applying this bound to 
Equation~\eqref{eq:whatu:sum1:a},
\begin{equation*} 
\sum_{s=1}^N \frac{ \gamma_n \tau_s^{-2} \nub{s}^{3/2} }
                { \sum_t \tau_t^{-1} }
\le \gamma_n \left( \max_s \tau_s^{-1} \nub{s} \right)
	\left(
	 \frac{ \sum_{s=1}^N  u_s/ \wopt_s  }
		{ \sum_{s=1}^N \nub{s}^{-1} } \right)^{1/2}
\end{equation*}
Applying our growth assumption from Equation~\eqref{eq:assum:harmonicish},
we conclude that
\begin{equation} \label{eq:whatu:sum1:final}
\sum_{s=1}^N \frac{ \gamma_n \tau_s^{-1} \nub{s}^{3/2} }
                { \tau_s \sum_t \tau_t^{-1} }
\le \left( \sum_{s=1}^N \nub{s}^{-1} \right)^{-1/2} \frac{ \log n }{ \log n + \log N }.
\end{equation}

Turning to the second sum on the right-hand side of
Equation~\eqref{eq:whatu:bound1},
another application of Equation~\eqref{eq:rhohatinv2tauinv}
to bound $\rhohat_s^{-1}$, followed by Equation~\eqref{eq:assum:ratiogrowth},
implies that
\begin{equation*} \begin{aligned}
\sum_{s=1}^N \sum_{t=1}^N
	\frac{ \gamma_n \nub{t} \nub{s}^{1/2} }
		{ \tau_s \tau_t^2 \left( \sum_t \rhohat_t^{-1} \right)
			\left( \sum_t \tau_t^{-1} \right) }
&\le C\gamma_n \left( \max_s \frac{ \nub{s} }{\tau_s} \right)
	\sum_{s=1}^N \sum_{t=1}^N
		\frac{ \tau_t^{-1} \tau_s^{-1} \nub{s}^{1/2} }
                      { \left( \sum_t \tau_t^{-1} \right)^2 } \\
&= C\gamma_n \left( \max_s \frac{ \nub{s} }{\tau_s} \right)
	\sum_{s=1}^N \frac{ \tau_s^{-1} \nub{s}^{1/2} }
                      { \sum_t \tau_t^{-1} }.
\end{aligned} \end{equation*}
Applying Equation~\eqref{eq:concavpayoff} once again, we have
\begin{equation*}
\sum_{s=1}^N \sum_{t=1}^N
        \frac{ \gamma_n \nub{t} \nub{s}^{1/2} }
                { \tau_s \tau_t^2 \left( \sum_t \rhohat_t^{-1} \right)
                        \left( \sum_t \tau_t^{-1} \right) }
\le C\gamma_n \left( \max_s \frac{ \nub{s} }{\tau_s} \right)
	\left( \frac{ \sum_{s=1}^N  u_s / \wopt_s } 
		{ \sum_t \nub{t}^{-1} } \right)^{1/2}.
\end{equation*}
Once again applying our growth assumption in
Equation~\eqref{eq:assum:harmonicish}, we have
\begin{equation*}
\sum_{s=1}^N \sum_{t=1}^N
        \frac{ \gamma_n \nub{t} \nub{s}^{1/2} }
                { \tau_s \tau_t^2 \left( \sum_t \rhohat_t^{-1} \right)
                        \left( \sum_t \tau_t^{-1} \right) }
\le 
C \left( \sum_{s=1}^N \nub{s}^{-1} \right)^{-1/2}
	\frac{ \log n }{ \log n + \log N }.
\end{equation*}
Applying this bound and Equation~\eqref{eq:whatu:sum1:final}
to Equation~\eqref{eq:whatu:bound1},
\begin{equation*}
\sum_{s=1}^N |\what_s-u_s| \nub{s}^{1/2}
\le
C \left( \sum_{s=1}^N \nub{s}^{-1} \right)^{-1/2}
	\frac{ \log n }{ \log n + \log N },
\end{equation*}
completing the proof.
\end{proof}

We are now ready to prove Theorem~\ref{thm:subgamma:XhatXopt}.

\begin{delayedproof}{Proof of Theorem~\ref{thm:subgamma:XhatXopt}}
We first establish the bound on $\| \Xopt - X \Vopt \|_{\tti}$.
Defining $\Aopt = \sum_w \wopt \AA{s}$, we have $\Xopt = \ASE( \Aopt, d )$.
Setting $w_s = \wopt_s$ for $s \in [N]$ in Theorem~\ref{thm:subexpXtilde},
Equation~\eqref{eq:paragrowth} holds
by virtue of our assumption in Equation~\eqref{eq:subgammagrowth}, and thus
\begin{equation} \label{eq:XoptXtti}
\| \Xopt - X \Vopt \|_{\tti}
\le
\frac{ Cd \log n }{ \lambda^{1/2}(P) \sqrt{ \sum_{s=1}^N \nub{s}^{-1} } }
+ \frac{ C d n \kappa(P) \log^2 n }{ \lambda_d^{3/2}(P) 
        	\sum_{s=1}^N \nub{s}^{-1} }.
\end{equation}

To prove the corresponding bound on $\| \Xhat - X V \|_{\tti}$,
write $\Ahat = \sum_s \what_s \AA{s}$ for ease of notation,
where $\{ \what_s \}_{s=1}^N$ are as in Equation~\eqref{eq:def:reminder:what}.
So long as for some $c_0 \in [0,1)$, 
\begin{equation} \label{eq:spectralrequirement}
\left\| \Ahat - P \right\| < c_0 \lambda_d(P) ~~~\text{ eventually,}
\end{equation}
Lemma~\ref{lem:genericbound} implies that
\begin{equation} \label{eq:XhatX:1}
\begin{aligned}
\| \Xhat &- X V \|_{\tti} \\
&\le
\frac{ \left\| \left( \Ahat - P \right) \UP \right\|_{\tti} }
     { \lambda_d^{1/2}(P) } 
+
\frac{ C \left\| \UP^T \left( \Ahat - P \right) \UP \right\|_F }
     { \lambda_d^{1/2}(P) } 
+ \frac{ C d \kappa(P) \left\| \Ahat  - P \right\|^2 }
     { \lambda_d^{3/2}(P) }
\end{aligned} \end{equation}
We will assume for now that Equation~\eqref{eq:spectralrequirement} holds,
and we will bound each of the norms on the right-hand
side of Equation~\eqref{eq:XhatX:1} in turn.

Noting that $\sum_s (\what_s - \wopt_s)P = 0$
and applying the triangle inequality,
\begin{equation} \label{eq:AhatP:ttitri}
\left\| \left( \Ahat-P \right) \UP \right\|_{\tti}
\le \sum_{s=1}^N \left| \what_s-\wopt_s \right|
		\left\| (\AA{s}-P) \UP \right\|_{\tti}
   + \left\| \left( \Aopt -P \right) \UP \right\|_{\tti}.
\end{equation}
By a slight adaptation of the arguments used to prove
the non-delocalized version of Proposition~\ref{prop:APtti},
it holds with probability $1-O(n^{-2})$ that for all $s \in [N]$,
\begin{equation} \label{eq:indivbound:tti}
\left\| (\AA{s}-P) \UP \right\|_{\tti}
\le C\sqrt{d} \nub{s}^{1/2}(\log N + \log n),
\end{equation}
where the $\log N$ term is included to permit a union bound over all
$s \in [N]$.
Similarly, the non-delocalized version of Proposition~\ref{prop:APtti}
implies that with probability at least $1-O(n^{-2})$,
\begin{equation*}
\left\| \left( \Aopt -P \right) \UP \right\|_{\tti}
\le \sqrt{d} \left( \sum_{s=1}^N \wopt_s^2 \nub{s} \right)^{1/2} \log n
= \sqrt{d} \left( \sum_{s=1}^N \nub{s}^{-1} \right)^{-1/2} \log n,
\end{equation*}
where we have plugged in the definition of $\wopt_s$.
Applying this and Equation~\eqref{eq:indivbound:tti}
to Equation~\eqref{eq:AhatP:ttitri},
it holds with probability at least $1-O(n^{-2})$ that
\begin{equation} \label{eq:AhatP:tti}
\begin{aligned}
&\left\| \left( \Ahat-P \right) \UP \right\|_{\tti} \\
&~~~~~~~~~\le C \sqrt{d} \sum_{s=1}^N \left| \what_s-\wopt_s \right|
			\nub{s}^{1/2}(\log N + \log n)
	+ \frac{ C\sqrt{d} \log n }{ \sqrt{ \sum_{s=1}^N \nub{s}^{-1} } }
\end{aligned} \end{equation}
By a similar argument,
this time using the non-delocalized version of Proposition~\ref{prop:UTAPU}
applied to $\AA{1},\AA{2},\dots,\AA{N}$ and $\Aopt$,
it holds with probability $1 - O(n^{-2})$ that
\begin{equation} \label{eq:AhatP:frob}
\begin{aligned}
& \left\| \UP^T \left( \Ahat-P \right) \UP \right\|_F \\
&~~~~~~~~~\le C d \sum_{s=1}^N \left| \what_s - \wopt_s \right|
			\nub{s}^{1/2}(\log N + \log n)
	+ \frac{ C d \log n }{ \sqrt{ \sum_{s=1}^N \nub{s}^{-1} } }.
\end{aligned} \end{equation}
Finally, applying Lemma~\ref{lem:mxspecbound} to $\Aopt$
and with $N=1$ to
each of $\AA{1},\AA{2},\dots,\AA{N}$ separately,
it holds with probability at least $1 - O(n^{-2})$ that
\begin{equation*}
\left\| \Aopt - P \right\|
\le C \left( \sum_{s=1}^N \wopt_s^2 \nub{s} \right)^{1/2} \sqrt{n} \log n
= \frac{ C \sqrt{n} \log n }{ \sqrt{ \sum_{s=1}^N \nub{s}^{-1} } }
\end{equation*}
and for all $s=1,2,\dots,N$,
\begin{equation*}
\left\| \AA{s} - P \right\| \le C \nub{s}^{1/2} \sqrt{n}
		\left( \log n + \log N \right),
\end{equation*}
with the $\log N$ terms again included to allow a union bound over
all $s \in [N]$.
Applying the triangle inequality,
\begin{equation} \label{eq:AhatP:spectral}
\left\| \Ahat - P \right\|
\le C\sqrt{n} \left(
	\sum_{s=1}^N \left| \what_s - \wopt_s \right| \nub{s}^{1/2}
	\left( \log n + \log N \right)
  + \frac{ \log n }{ \sqrt{ \sum_{s=1}^N \nub{s}^{-1} } } \right)
\end{equation}

Plugging the bounds in Equations~\eqref{eq:AhatP:tti},~\eqref{eq:AhatP:frob}
and~\eqref{eq:AhatP:spectral} into Equation~\eqref{eq:XhatX:1},
we have that with probability $1-O(n^{-2})$,
\begin{equation*} \begin{aligned}
\| \Xhat &- X V \|_{\tti} \\
&\le
\frac{ C \sqrt{d} }{ \lambda_d^{1/2}(P) }
	\left(
	\sum_{s=1}^N \left| \what_s-\wopt_s \right|
			\nub{s}^{1/2}(\log N + \log n)
	+ \frac{ \log n }{ \sqrt{ \sum_{s=1}^N \nub{s}^{-1} } } \right) \\
&~~~+
\frac{ C d}{ \lambda_d^{1/2}(P) }
\left( \sum_{s=1}^N \left| \what_s - \wopt_s \right|
			\nub{s}^{1/2}(\log N + \log n)
	+ \frac{ \log n }{ \sqrt{ \sum_{s=1}^N \nub{s}^{-1} } } \right) \\
&~~~+ \frac{ C d n \kappa(P) } { \lambda_d^{3/2}(P) }
  \left( \sum_{s=1}^N \left| \what_s - \wopt_s \right| \nub{s}^{1/2}
        	\left( \log n + \log N \right)
  + \frac{ \log n }{ \sqrt{ \sum_{s=1}^N \nub{s}^{-1} } } \right)^2.
\end{aligned} \end{equation*}
Collecting terms, 
\begin{equation*} 
\begin{aligned}
\| &\Xhat - X V \|_{\tti} \\
&\le \frac{ C d}{ \lambda_d^{1/2}(P) }
\left( \sum_{s=1}^N \left| \what_s - \wopt_s \right|
                        \nub{s}^{1/2}(\log N + \log n)
        + \frac{ \log n }{ \sqrt{ \sum_{s=1}^N \nub{s}^{-1} } } \right) \\
&~~~+ \frac{ C d n \kappa(P) } { \lambda_d^{3/2}(P) }
  \left( \sum_{s=1}^N \left| \what_s - \wopt_s \right| \nub{s}^{1/2}
                \left( \log n + \log N \right)
  + \frac{ \log n }{ \sqrt{ \sum_{s=1}^N \nub{s}^{-1} } } \right)^2.
\end{aligned} \end{equation*}
Comparing this bound with that in Equation~\eqref{eq:XoptXtti},
it will suffice for us to show that for all suitably large $n$,
\begin{equation} \label{eq:nubgrowth:goal1}
\sum_{s=1}^N \left| \what_s - \wopt_s \right|
                        \nub{s}^{1/2}(\log N + \log n)
	\le C \frac{ \log n }{ \sqrt{ \sum_{s=1}^N \nub{s}^{-1} } }.
\end{equation}
Note that applying Equation~\eqref{eq:nubgrowth:goal1}
to Equation~\eqref{eq:AhatP:spectral},
followed by our growth assumption in Equation~\eqref{eq:subgammagrowth}, 
will imply that
\begin{equation*}
\| \Ahat - P \|
\le \frac{ C \sqrt{n} \log n }{ \sqrt{ \sum_{s=1}^N \nub{s}^{-1} } }
= o\left( \lambda_d(P) \right),
\end{equation*}
so that Equation~\eqref{eq:spectralrequirement} holds eventually.
Thus, our proof will be complete once we establish
Equation~\eqref{eq:nubgrowth:goal1}.

Applying the triangle inequality,
\begin{equation} \label{eq:whatwopt:triangle}
\sum_{s=1}^N \left| \what_s - \wopt_s \right| \nub{s}^{1/2}
\le
\sum_{s=1}^N \left| \what_s - u_s \right| \nub{s}^{1/2}
+
\sum_{s=1}^N \left| u_s - \wopt_s \right| \nub{s}^{1/2}.
\end{equation}
By Proposition~\ref{prop:whatu}, with probability $1-O(n^{-2})$,
the first of these two sums is bounded by
\begin{equation} \label{eq:whatu:repeat}
\sum_{s=1}^N \left| \what_s - u_s \right| \nub{s}^{1/2}
\le
\frac{ C \log n }{ \log n + \log N }
\left( \sum_{s=1}^N \nub{s}^{-1} \right)^{-1/2},
\end{equation}
The second sum on the right-hand side of Equation~\eqref{eq:whatwopt:triangle}
is bounded by
\begin{equation*} \begin{aligned}
\sum_{s=1}^N \left| u_s - \wopt_s \right| \nub{s}^{1/2}
&= \sum_{s=1}^N \wopt_s \left| 1 - \frac{ u_s }{ \wopt_t } \right|
			\nub{s}^{1/2} \\
&\le
\sqrt{ \sum_{s=1}^N \wopt_s \left(1 - \frac{u_s}{\wopt_s}\right)^2 }
\sqrt{ \sum_{s=1}^N \wopt_s \nub{s} },
\end{aligned} \end{equation*}
where the inequality follows from Cauchy-Schwarz.
Noting that
\begin{equation*}
\wopt_s \nub{s} = \frac{1}{ \sum_t \nub{t}^{-1} },
\end{equation*}
it follows that
\begin{equation*}
\sum_{s=1}^N \left| u_s - \wopt_s \right| \nub{s}^{1/2}
\le \sqrt{ \frac{ N \sum_{s=1}^N \wopt_s \left(1 - \frac{u_s}{\wopt_s}\right)^2 }
		{ \sum_t \nub{t}^{-1} } }.
\end{equation*}
Applying this and Equation~\eqref{eq:whatu:repeat}
to Equation~\eqref{eq:whatwopt:triangle}, we have
\begin{equation*}
\sum_{s=1}^N \left| \what_s - \wopt_s \right| \nub{s}^{1/2}
\le
\frac{ 1 }{ \sqrt{ \sum_{s=1}^N \nub{s}^{-1} } }
\left(
\frac{ C \log n }{ \log n + \log N }
+ \sqrt{ N \sum_{s=1}^N \wopt_s \left(1 - \frac{u_s}{\wopt_s}\right) }
\right).
\end{equation*}
Applying the growth assumption in Equation~\eqref{eq:assum:harmonicish2}
yields
\begin{equation*}
\sum_{s=1}^N \left| \what_s - \wopt_s \right| \nub{s}^{1/2}
\le
\frac{ C \log n}
	{ \left( \log n + \log N \right) \sqrt{ \sum_{s=1}^N \nub{s}^{-1} } }.
\end{equation*}
Multiplying by $(\log n + \log N)$
establishes Equation~\eqref{eq:nubgrowth:goal1},
completing the proof.
\end{delayedproof}

\newpage 
\bibliographystyle{plainnat}
\bibliography{biblio}

\end{document}